\documentclass{amsart}%
\usepackage{amsfonts, amsmath, amssymb}%
\usepackage{graphicx, epic, enumerate}
\usepackage{float}

\newtheorem{theorem}{Theorem}
\theoremstyle{definition}

\newtheorem{corollary}{Corollary}
\newtheorem{definition}{Definition}
\newtheorem{example}{Example}
\newtheorem{lemma}{Lemma}
\newtheorem{proposition}{Proposition}
\newtheorem{remark}{Remark}

\numberwithin{equation}{section}

\newcommand{\OR}{\overline{R}}

\begin{document}
\allowdisplaybreaks
\title{On the Kauffman-Jones polynomial for virtual singular links}

\author{Carmen Caprau}
\address{Department of Mathematics, California State University-Fresno, 5245 North Backer Avenue, M/S PB 108, Fresno, CA 93740, USA}
\email{ccaprau@csufresno.edu}
\urladdr{http://zimmer.csufresno.edu/~ccaprau}
\author{Kelsey Friesen}
\address{Mathematics Department, Reedley College, 995 North Reed Ave, Reedley, CA 93654, USA}
\email{kelsey.friesen@reedleycollege.edu}

\date{}
\thanks{This work was supported by a grant from the Simons Foundation ($\#355640$, Carmen Caprau).}

\begin{abstract} 
We extend the Kamada-Miyazawa polynomial to virtual singular links, which is valued in $\mathbb{Z}[A^2, A^{-2}, h]$. The decomposition of the resulting polynomial into two components, one in $\mathbb{Z}[A^2, A^{-2}]$ and the other in $\mathbb{Z}[A^2, A^{-2}]h$ yields the decomposition of the Kauffman-Jones polynomial of virtual singular links into two components, one in $\mathbb{Z}[A^2, A^{-2}]$ and the other in $\mathbb{Z}[A^2, A^{-2}]A^2$, where both components are invariants for virtual singular links.
 \end{abstract}

\maketitle
\section{Introduction}\label{sec:intro}
A \textit{virtual singular link diagram} is a decorated immersion of $k$ ($k \in \mathbb{N}$) disjoint copies of $S^1$ into $\mathbb{R}^2$, with finitely many transverse double points each of which has information of over/under, singular, and virtual crossings, as indicated in Figure~\ref{fig:crossings}. If $k = 1$, we have a \textit{virtual singular knot diagram}, or equivalently, a virtual singular link diagram of one component. The over/under markings indicate the classical crossings. A filled in circle is used to represent a \textit{singular crossing}. \textit{Virtual crossings} are represented by placing a small circle around the point where the two arcs meet transversely. 

\begin{figure}[ht]
\[ \raisebox{0cm}{\includegraphics[height=0.5in]{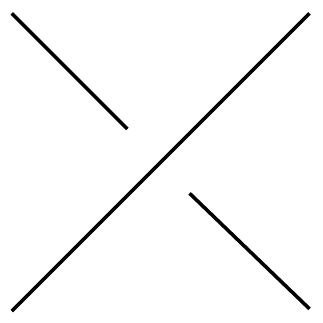}}  \hspace{.5in} \raisebox{0cm}{\includegraphics[height=0.5in]{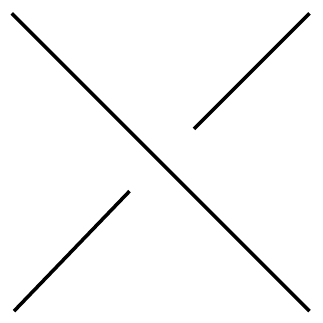}} \hspace{.5in}  \raisebox{0cm}{\includegraphics[height=0.5in]{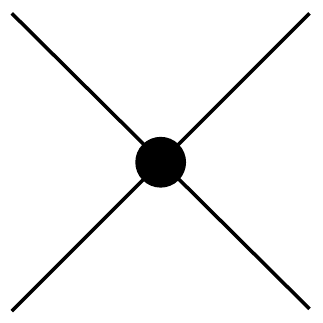}} \hspace{.5in} \raisebox{-.1cm}{\includegraphics[height=0.5in]{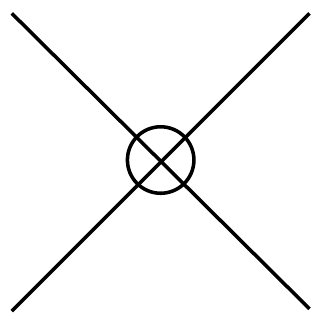}} \]
\caption{Types of crossings} \label{fig:crossings}
\end{figure}

Unless otherwise specified, we will use the term `link' to refer to both knots and links, for simplicity.

Similar to the case of virtual knot theory, there is a useful topological interpretation for virtual singular knot theory in terms of embeddings of singular links in thickened surfaces $S_g \times I$, where $S_g$ is a compact oriented surface of genus $g$ and $I$ is the unit interval. A diagram of a virtual singular link in $S_g \times I$ is then drawn on $S_g$, and virtual crossings are merely byproducts of the projection of the immersion into the surface $S_g$ composed with the projection of the surface $S_g$ into the plane $\mathbb{R}^2$.

Two virtual singular link diagrams are said to be \textit{equivalent} (or \textit{ambient isotopic}) if they are related by a finite sequence of the \textit{extended Reidemeister moves} depicted in Figure~\ref{fig:isotopies} (where only one possible choice of crossings is indicated in the diagrams). The move $RS_2$ exemplifies that the singular crossings can be regarded as rigid disks. Equivalently, there is a fixed ordering of the four strands meeting at a singular crossing and the cyclic ordering is determined via the rigidity of the disk. The set of moves in Figure~\ref{fig:isotopies} (where all possible crossing types for the classical crossings need to be considered) defines an equivalence relation on the set of virtual singular link diagrams. A \textit{virtual singular link} (or \textit{virtual singular link-type}) is then the equivalence class of a virtual singular link diagram. 

\begin{figure}[ht]
\[
 \raisebox{-.5cm}{\includegraphics[height=.4in]{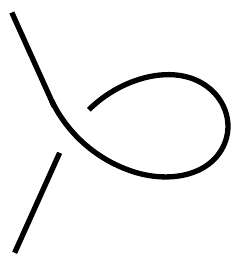}} \hspace{.2cm}  \stackrel{R_1}{\longleftrightarrow}  \hspace{.2cm} \raisebox{-.5cm}{\includegraphics[height=.4in]{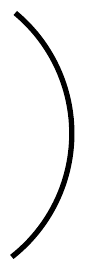}}  \hspace{1cm}  
 \raisebox{-.5cm}{\includegraphics[height=.4in]{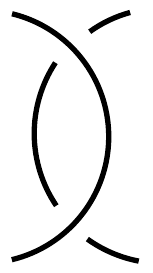}} \hspace{.2cm} \stackrel{R_2}{\longleftrightarrow} \hspace{.2cm} \raisebox{-.5cm}{\includegraphics[height=.4in]{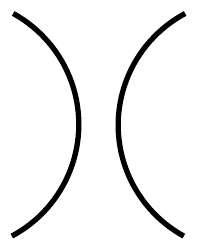}}  \hspace{1 cm}
 \raisebox{-.5cm}{\includegraphics[height=.4in]{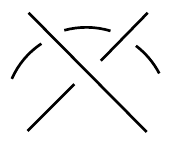}} \hspace{.2cm}  \stackrel{R_3}{\longleftrightarrow} \hspace{.2cm} \raisebox{-.5cm}{\includegraphics[height=.4in]{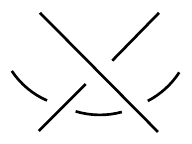}} 
\]  
\[  \raisebox{-10pt}{\includegraphics[height=.4in]{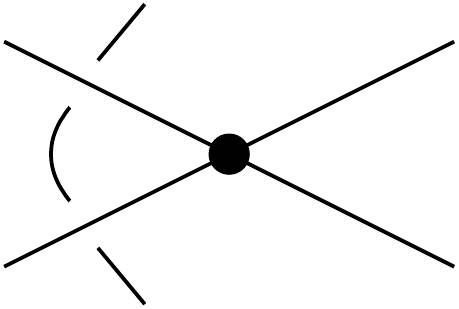}}\,\, \stackrel{RS_1}{\longleftrightarrow}
 \,\, \raisebox{-10pt}{\includegraphics[height=.4in]{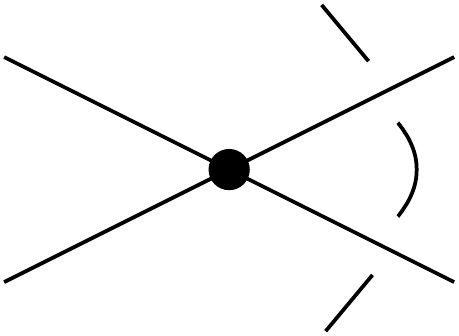}}  \hspace{1.5cm}
\raisebox{-10pt}{\includegraphics[height=.35in]{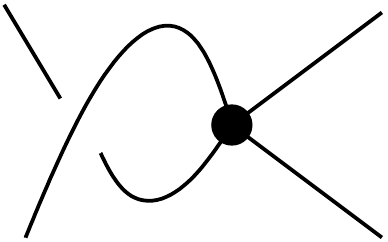}}\,\, \stackrel{RS_2}{\longleftrightarrow} \,\, \raisebox{15pt}{\includegraphics[height=.35in, angle = 180]{N5-left1}} \]
\[ \raisebox{-.5cm}{\includegraphics[height=.5in]{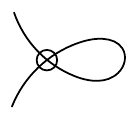}} \hspace{.2cm}  \stackrel{V_1}{\longleftrightarrow} \hspace{.2cm} \raisebox{-.5cm}{\includegraphics[height=.5in]{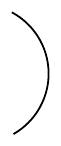}}  \hspace{2cm} 
\raisebox{-.5cm}{\includegraphics[height=.5in]{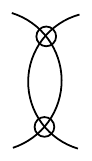}} \hspace{.2cm}  \stackrel{V_2}{\longleftrightarrow} \hspace{.2cm} \raisebox{-.5cm}{\includegraphics[height=.5in]{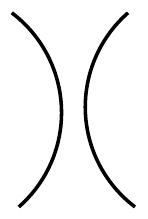}} \]
\[ \raisebox{-.5cm}{\includegraphics[height=.5in]{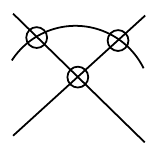}} \hspace{.2cm} \stackrel{V_{3-v}}{\longleftrightarrow}  \hspace{.2cm} 
\raisebox{-.5cm}{\includegraphics[height=.5in]{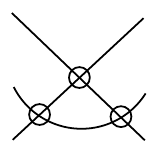}} \hspace{0.5cm} \raisebox{-.5cm}{\includegraphics[height=.5in]{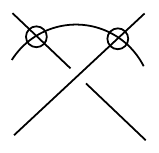}} \hspace{.2cm}   \stackrel{V_{3-c}}{\longleftrightarrow} \hspace{.2cm} \raisebox{-.5cm}{\includegraphics[height=.5in]{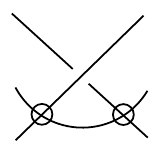}} 
 \hspace{0.5cm}
\raisebox{-.5cm}{\includegraphics[height=.5in]{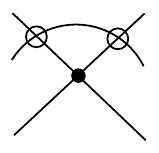}} \hspace{.2cm}   \stackrel{V_{3-s}}{\longleftrightarrow} \hspace{.2cm} \raisebox{-.5cm}{\includegraphics[height=.5in]{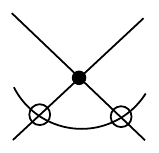}} 
\]

\caption{The extended Reidemeister moves} \label{fig:isotopies}
\end{figure}

In addition to the set of extended Reidemeister moves, there is a set of moves known as the \textit{forbidden moves} shown in Figure~\ref{fig:forbidden}. Although these moves look similar to those previously shown, they do not represent isotopic virtual singular link diagrams.

\begin{figure}[ht]
\[ \raisebox{-.5cm}{\includegraphics[height=.45in]{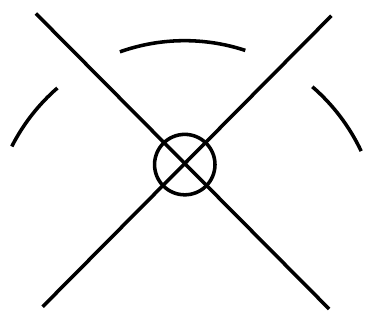}} \hspace{.2cm}   \longleftrightarrow  \hspace{.2cm}   \raisebox{-.5cm}{\includegraphics[height=.45in]{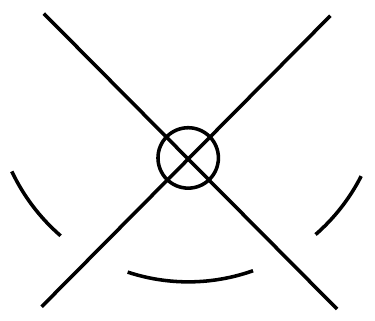}} \hspace{1.5cm}  \raisebox{-.5cm}{\includegraphics[height=.45in]{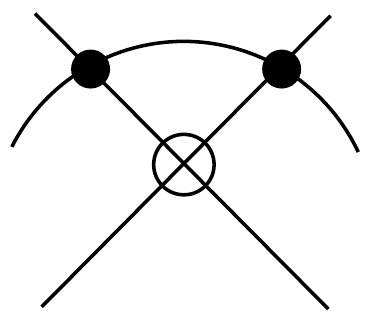}} \hspace{.2cm}   \longleftrightarrow  \hspace{.2cm}   \raisebox{-.5cm}{\includegraphics[height=.45in]{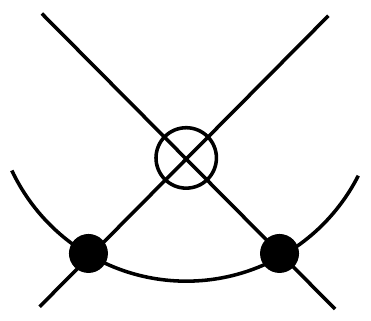}} \]

\[ \raisebox{-.5cm}{\includegraphics[height=.45in]{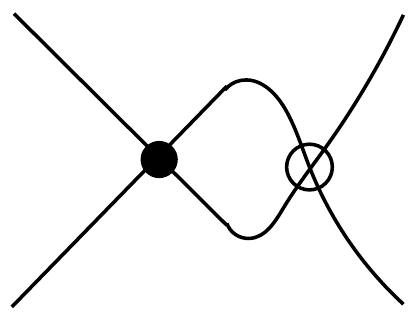}}  \hspace{.2cm}  \longleftrightarrow  \hspace{.2cm}  \raisebox{-.5cm}{\includegraphics[height=.45in]{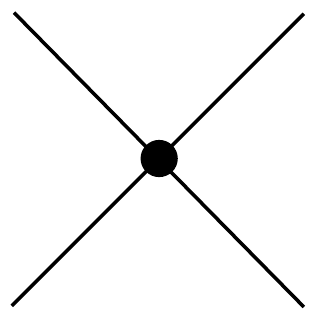}}  \hspace{.2cm}  \longleftrightarrow  \hspace{.2cm}  \raisebox{.7cm}{\includegraphics[height=.45in, angle=180]{f32}}  \]
\caption{The forbidden moves for virtual singular knot diagrams} \label{fig:forbidden}
\end{figure}

Throughout this paper, we work with oriented virtual singular links. The extended Reidemeister moves in Figure \ref{fig:isotopies} are still used to relate two equivalent virtual singular link diagrams, just now with given orientations on the strands, where all possible choices of orientation are considered and where the fixed orientation on the strands for the two diagrams in a certain move must agree.

In ~\cite{Ka2}, L. Kauffman constructed a polynomial invariant for oriented virtual links, denoted $f_L(A)$, which is an extension from classical links to virtual links of the polynomial $f[L] (A)$ constructed by the same author in~\cite{Ka1}. Recall that the polynomial $f[L] (A)$ provides a state model for the Jones polynomial~\cite{Jo}. We refer to the polynomials $f_L(A)$ and $f[L] (A)$ as the \textit{Kauffman-Jones polynomials} for virtual links and classical links, respectively. 

The goal of this paper is to go one step further and extend the polynomial $f_L(A)$ to virtual singular links. In doing so, we provide two definitions for our polynomial, which we denote by $\left < L \right>$, where $L$ is a virtual singular link. The first definition uses skein relations, while the second definition provides a state summation formula for the new polynomial $\left < L\right>$. Using a state-sum formula, N. Kamada's and Y. Miyazawa's proved in~\cite{KM} that the Kauffman-Jones polynomial for virtual links splits non-trivially with respect to the powers of $A$ modulo four. In this paper, we use primarily skein relations to show that the same holds for our polynomial.  Specifically, we show that if $L$ is a virtual singular link with $k$ components, then $\left < \,L\, \right > = \phi(L) + \psi(L)$,
where  $\phi(L) \in \mathbb{Z}[A^4,A^{-4}] \cdot A^{2(k-1)}$ and $\psi(L) \in \mathbb{Z}[A^4,A^{-4}] \cdot A^{2k} $.
Both polynomials $\phi(L)$ and $\psi(L)$ are invariants for virtual singular links, as is $\left < L \right>$. Our results are similar in spirit to those by Kamada and Miyazawa, but the approach and proofs in our paper are different from those in~\cite{KM}, in that we make intensive use of our defining skein relations for the polynomial $\left < L \right>$, while~\cite{KM} relies exclusively on state-sum formulas.

To help the reader, we tried to have a self-contained exposition. The paper is organized as follows:  Given a virtual singular link $L$, we introduce our version of the Kauffman-Jones polynomial $\left < L \right> \in \mathbb{Z}[A^2,A^{-2}]$ in Section~\ref{sec:poly-def}. For this, we use skein relations and a type of graphs which we call purely virtual magnetic graphs. In Section~\ref{sec:second-poly} we define a two-variable polynomial $R(L) \in \mathbb{Z}[A^2,A^{-2}, h]$, which is an ambient isotopy invariant for virtual singular links $L$ and, when restricted to $h = 1$, equals the polynomial $\left < L \right>$. In Section~\ref{sec:split} we show that for any virtual singular link $L$ with $k$ components, the polynomial $R(L)$ can be written as $R(L)=\phi(L)h + \psi(L)$, where $\phi(L) \in \mathbb{Z}[A^4,A^{-4}] \cdot A^{2(k-1)}$ and $\psi(L) \in \mathbb{Z}[A^4,A^{-4}] \cdot A^{2k}$. It follows that the Kauffman-Jones polynomial $\left < L \right>$ for virtual singular links splits into two parts with respect to the powers of $A$ modulo 4.

\section{Kauffman-Jones polynomial for virtual singular links}\label{sec:poly-def}

In this section, we provide our approach in defining the Kauffman-Jones polynomial for virtual singular links.

\begin{definition}
We call a \textit{purely virtual magnetic graph} an immersed directed graph in $\mathbb{R}^2$ with bivalent vertices such that the edges are oriented alternately, as shown in Figure \ref{fig:alt_edges}(a), and with self-intersections represented as virtual crossings. Equivalently, each bivalent vertex is either a $\textit{sink}$, meaning the edges are oriented towards the vertex, or a \textit{source}, where the edges are oriented away from the vertex.
\end{definition}
We do allow for components in a purely virtual magnetic graph to consist of oriented closed loops without vertices, as shown in Figure \ref{fig:alt_edges}(b).

\begin{figure}[ht] 
\[
\put(70, -7){\fontsize{9}{11}(a)} \put(240, -7){\fontsize{9}{11}(b)}
\raisebox{.3in}{\includegraphics[height=.17in]{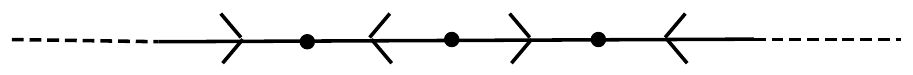}} \hspace{2cm} \includegraphics[height=1.2in]{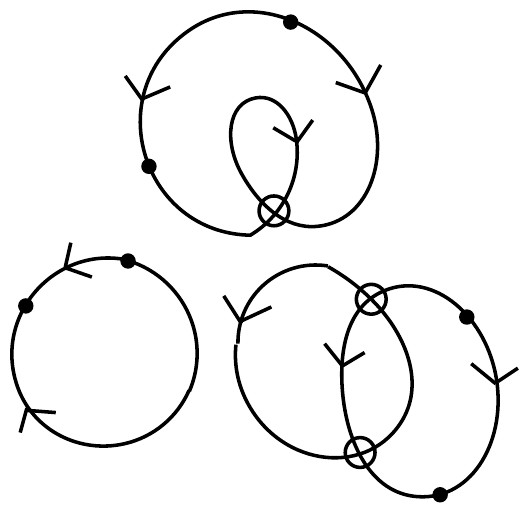} 
 \]
\caption{(a) Alternating oriented edges; (b) A purely virtual magnetic graph} \label{fig:alt_edges}
\end{figure}

We remark that N. Kamada and Y. Miyazawa used in their work~\cite{KM} the term ``magnetic graph diagrams" to refer to such graphs.

Given a virtual singular link $L$ with diagram $D$, we resolve all of the classical and singular crossings in $D$ in two ways, as shown in Figure~\ref{fig:or-dis resolutions}, leaving the virtual crossings  in place. We refer to the two resolutions of a crossing as the \textit{oriented} and, respectively, \textit{disoriented resolutions}.

\begin{figure}[ht] 
\[
\includegraphics[height=.5in]{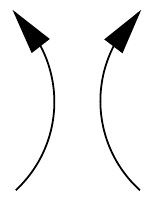} \hspace{1in} \includegraphics[height=.5in]{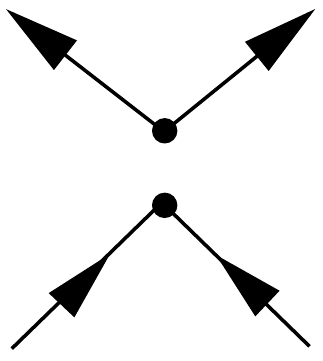}
\]
\caption{Oriented and disoriented resolutions}\label{fig:or-dis resolutions}
\end{figure}

By resolving each of the classical and singular crossings in the diagram $D$ in one of the two ways as shown above, we obtain a \textit{Kauffman-Jones state} associated with $D$. That is, a Kauffman-Jones state is a purely virtual magnetic graph. The Kauffman-Jones states of $D$ will receive certain weights, which are polynomials in $\mathbb{Z}[A^2, A^{-2}]$, where these weights are uniquely determined by the skein relations depicted in Figure~\ref{fig:extended-jones}. The Kauffman-Jones polynomial of $D$ is then a Laurent polynomial, denoted $\left< D \right>$, defined as a formal $\mathbb{Z}[A^2, A^{-2}]$-linear combination of the evaluations of all of the Kauffman-Jones states associated with $D$.

We remind the reader that the diagrams in both sides of a skein relation represent parts of larger diagrams that are identical, except near a point where they look as indicated in the skein relation. 

\begin{figure}[ht]
 \[ \left< \, \raisebox{-12pt}{\includegraphics[height=.4in]{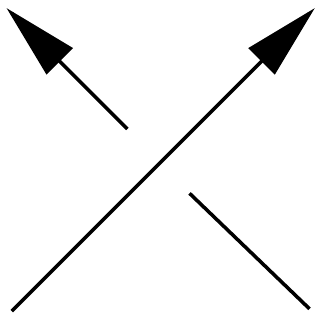}}\, \right> = -A^{-2} \left< \,  \raisebox{-12pt}{\includegraphics[height=.4in]{oriented}} \, \right> -  A^{-4} \left< \,  \raisebox{-12pt}{\includegraphics[height=.4in]{unoriented}} \, \right> \] 

 \[ \left< \, \raisebox{-12pt}{\includegraphics[height=.4in]{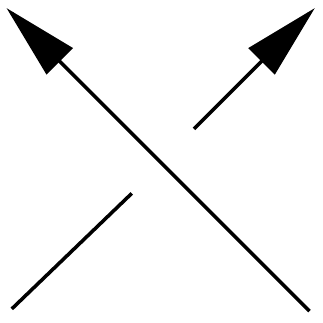}}\, \right> = -A^{2} \left< \, \raisebox{-12pt}{\includegraphics[height=.4in]{oriented}} \, \right> -  A^{4} \left< \, \raisebox{-12pt}{\includegraphics[height=.4in]{unoriented}} \, \right> \] 

\[ 
\left < \, \raisebox{-12pt}{\includegraphics[height=.4in]{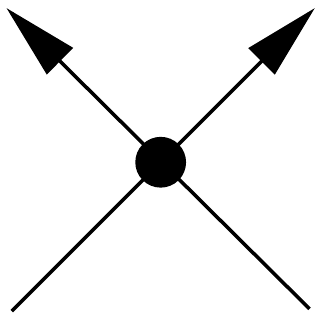}} \,\right > = (-A^{2}-A^{-2}) \left < \, \raisebox{-12pt}{\includegraphics[height=.4in]{oriented}}\, \right > + (-A^4-A^{-4}) \left < \, \raisebox{-12pt}{\includegraphics[height=.4in]{unoriented}}\, \right >
\]
\caption{Skein relations for classical and singular crossings} \label{fig:extended-jones}
\end{figure}

The Kauffman-Jones states associated with $D$ are uniquely evaluated using the \textit{graph skein relations} depicted in Figure~\ref{fig:graph-skein-rel}. Note that the first set of relations in Figure~\ref{fig:graph-skein-rel} say that we can remove or introduce pairs of adjacent bivalent vertices without changing the evaluation of a state, as long as the two vertices are oppositely oriented (one is a source and the other is a sink). Moreover, the last skein relation says that the Kauffman-Jones polynomial is multiplicative with respect to disjoint unions. Finally, any closed loop with or without virtual crossings is evaluated to $-A^2-A^{-2}$. Equivalently, Kauffman-Jones states are regarded up to the \textit{pure virtual moves} $V_1$, $V_2$ and $V_{3-v}$.

\begin{figure}
 \[
\left < \raisebox{-.3cm}{\includegraphics[height=.25in]{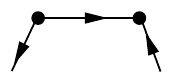}} \right > = \left < \raisebox{-.3cm}{\includegraphics[height=.25in]{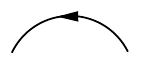}} \right >  \hspace{1cm} \left < \raisebox{-.3cm}{\includegraphics[height=.25in]{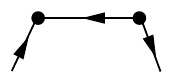}} \right > = \left < \raisebox{-.3cm}{\includegraphics[height=.25in]{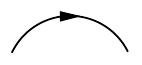}} \right >
\]

 \[
\left < \, \raisebox{-.3cm}{\includegraphics[height=.3in]{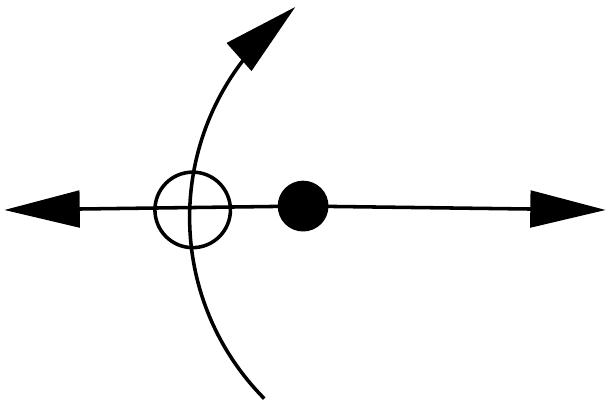}} \, \right > = \left < \, \raisebox{-.3cm}{\includegraphics[height=.3in]{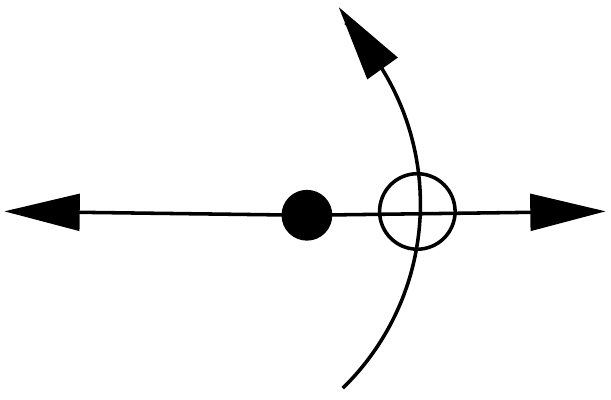}} \, \right >  \hspace{1cm} \left < \, \raisebox{-.3cm}{\includegraphics[height=.3in]{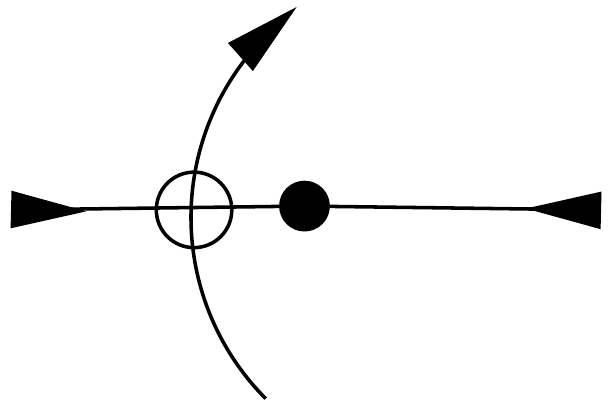}} \, \right > = \left <  \, \raisebox{-.3cm}{\includegraphics[height=.3in]{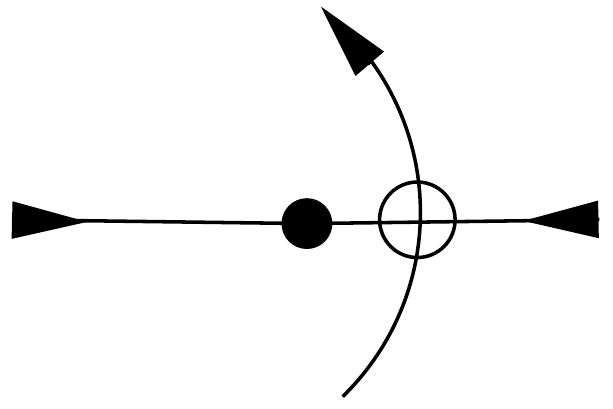}} \, \right >
\]

\[ \left < \raisebox{-.2cm}{\includegraphics[height=.3in]{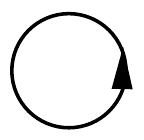}} \right > = -A^2-A^{-2} = \left < \raisebox{-.2cm}{\includegraphics[height=.3in]{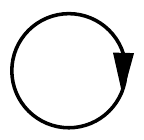}} \right >
\]
\[
  \left < \Gamma \cup \raisebox{-.2cm}{\includegraphics[height=.25in]{pspin}} \right > = (-A^2 -A^{-2} ) \left < \Gamma \right > = \left < \Gamma \cup \raisebox{-.2cm}{\includegraphics[height=.25in]{nspin}} \right >
\]
\caption{Graph skein relations for evaluating Kauffman-Jones states} \label{fig:graph-skein-rel}
\end{figure}

We note that, up to the first two sets of the graph skein relations and the pure virtual moves ($V_1, V_2$ and $V_{3-v}$), each Kauffman-Jones state is equivalent to a disjoint collection of directed circles in the plane.

A quick analysis of the defining skein relations for classical and singular crossings imply that the following skein relation holds, which will come in handy in proofs:
\begin{eqnarray}\label{eq:sing-crossing}
\left < \,\raisebox{-12pt}{\includegraphics[height=.4in]{osingular}} \,\right > = \left < \, \raisebox{-12pt}{\includegraphics[height=.4in]{opos}} \,\right > + \left < \,\raisebox{-12pt}{\includegraphics[height=.4in]{oneg}} \,\right >
\end{eqnarray}

\begin{theorem}\label{thm:inv}
The Kauffman-Jones polynomial $\left< \,\cdot \, \right>$ as defined by the skein relations in Figures~\ref{fig:extended-jones} and ~\ref{fig:graph-skein-rel}  is an ambient isotopy invariant for virtual singular links.
\end{theorem}

\begin{proof}
We need to show that $\left< \,\cdot \, \right>$ is invariant under the extended Reidemeister moves  for virtual singular link diagrams given in Figure~\ref{fig:isotopies}.

By our definition of the evaluation of Kauffman-Jones states, we have that $\left< \,\cdot \, \right>$ is invariant under the virtual moves $V_1$, $V_2$ and $V_{3-v}$.

We consider first the Reidemeister move $R_1$ involving a positive classical crossing; the case of a negative classical crossing follows the same format, and thus we omit it here. We use the appropriate skein relation for the positive crossing, followed by the first and last graph skein relations for evaluating Kauffman-Jones states, as shown below:
\begin{eqnarray*}
\left< \, \raisebox{-.3cm}{\includegraphics[height=.3in]{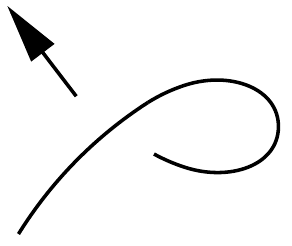}}\, \right> &=& -A^{-2} \left<\, \raisebox{-.3cm}{\includegraphics[height=.3in]{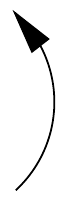}} \raisebox{-.2cm}{\includegraphics[height=.25in]{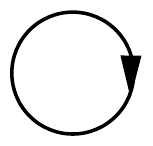}}\, \right> -A^{-4} \left<\, \raisebox{-.3cm}{\includegraphics[height=.3in]{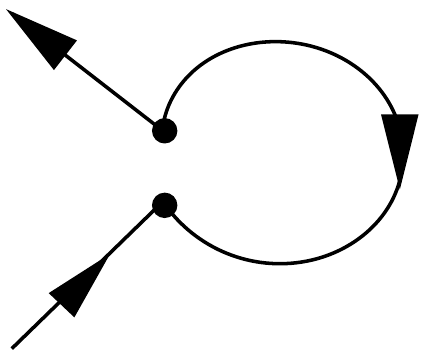}}\, \right> \\
&=& -A^{-2}(-A^2-A^{-2}) \left<\, \raisebox{-.3cm}{\includegraphics[height=.3in]{larc}}\, \right> -A^{-4} \left<\, \raisebox{-.3cm}{\includegraphics[height=.3in]{larc}}\, \right> \\
&=& (1+A^{-4}-A^{-4}) \left< \,\raisebox{-.3cm}{\includegraphics[height=.3in]{larc}} \,\right> \\
&=& \left< \,\raisebox{-.3cm}{\includegraphics[height=.3in]{larc}}\, \right>.
\end{eqnarray*}

 Next, we consider the Reidemeister move $R_2$ where both strands are oriented similarly, say upwards. We begin by resolving the top negative classical crossing:
 \begin{eqnarray*}
\left< \,\raisebox{-.4cm}{\includegraphics[height=.4in]{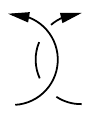}}\, \right> 
&=& -A^{2} \left<\, \raisebox{-.4cm}{\includegraphics[height=.4in]{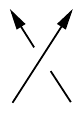}}\, \right> - A^{4} 
\left< \,\raisebox{-.4cm}{\includegraphics[height=.4in]{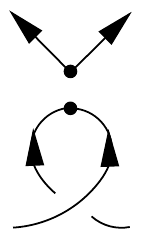}} \,\right> .
\end{eqnarray*} 
We then resolve the positive classical crossing in each of the resulting diagrams to obtain a linear combination of evaluations of Kauffman-Jones states associated with the original diagram, and then employ any necessary graph skein relations:
\begin{eqnarray*}
\left< \,\raisebox{-.4cm}{\includegraphics[height=.4in]{pr21}}\, \right> &=& -A^2 \left( -A^{-2} \left< \,\raisebox{-.4cm}{\includegraphics[height=.4in]{oriented}}\, \right> -A^{-4} 
\left< \,\raisebox{-.4cm}{\includegraphics[height=.4in]{unoriented}} \,\right> \right) \\
 & &- A^4 \left( -A^{-2} \left<\, \raisebox{-.4cm}{\includegraphics[height=.4in]{unoriented}}\, \right> -A^{-4} \left< \, \raisebox{-.5cm}{\includegraphics[height=.45in, width = .2in]{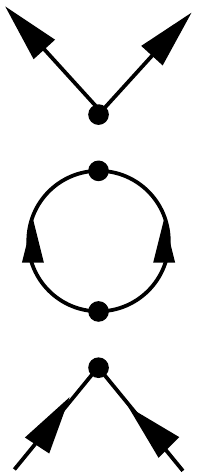}} \, \right> \right) \\
&=& \left< \,\raisebox{-.4cm}{\includegraphics[height=.4in]{oriented}}\, \right> + (A^2+A^{-2}) \left< \,\raisebox{-.4cm}{\includegraphics[height=.4in]{unoriented}}\, \right> + 1\cdot (-A^2-A^{-2}) \left<\, \raisebox{-.4cm}{\includegraphics[height=.4in]{unoriented}}\, \right> \\
&=& \left< \,\raisebox{-.4cm}{\includegraphics[height=.4in]{oriented}} \,\right>.
\end{eqnarray*}

The invariance under the Reidemeister move $R_2$ with oppositely oriented strands is verified below:
\begin{eqnarray*}
\left< \,\raisebox{-.4cm}{\includegraphics[height=.4in]{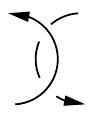}}\, \right> &=& -A^{-2} 
\left<\, \raisebox{-.4cm}{\includegraphics[height=.4in]{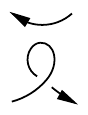}}\, \right> - A^{-4} 
\left<\, \raisebox{-.4cm}{\includegraphics[height=.4in]{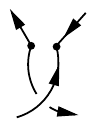}}\, \right> \\
&=& -A^{-2} \left( -A^{2} \left<\, \raisebox{-.4cm}{\includegraphics[height=.4in]{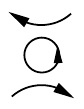}}\, \right> -A^{4} 
\left<\, \raisebox{-.4cm}{\includegraphics[height=.4in]{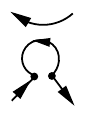}}\, \right> \right) \\
 & &- A^{-4} \left( -A^{2} \left< \,\raisebox{-.4cm}{\includegraphics[height=.4in]{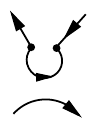}} \,\right> -A^{4} \left< \,\raisebox{-.4cm}{\includegraphics[height=.4in]{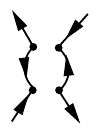}} \,\right> \right) \\
&=& 1\cdot (-A^2-A^{-2}) \left<\, \raisebox{-.5cm}{\includegraphics[height=.4in]{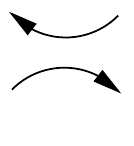}}\, \right> + (A^2+A^{-2}) \left<\, \raisebox{-.5cm}{\includegraphics[height=.4in]{pr28op}}\, \right> +  \left< \,\raisebox{-.4cm}{\includegraphics[height=.4in]{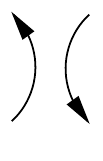}} \,\right>  \\
&=& \left<\, \raisebox{-.5cm}{\includegraphics[height=.4in]{pr29op}} \,\right>.
\end{eqnarray*}

Next, we consider the Reidemeister move $R_3$. Resolving one of the classical crossings in each diagram, and using that the polynomial is invariant under the Reidemeister move $R_2$, results in the following equalities:
\begin{eqnarray*}
\left< \,\raisebox{-.4cm}{\includegraphics[height=.4in]{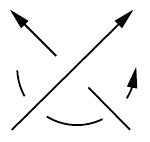}}\, \right> &=& -A^2 \left< \,\raisebox{-.4cm}{\includegraphics[height=.4in]{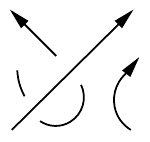}}\, \right>  -A^4 \left< \,\raisebox{-.4cm}{\includegraphics[height=.4in]{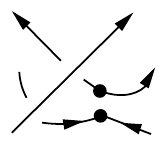}}\, \right> \\
&\stackrel{R_2}{=}& -A^2 \left< \,\raisebox{-.4cm}{\includegraphics[height=.4in]{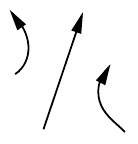}}\, \right> -A^4 \left<\, \raisebox{-.4cm}{\includegraphics[height=.4in]{pr33}}\, \right>,
\end{eqnarray*}
and
\begin{eqnarray*}
\left< \, \raisebox{-.4cm}{\includegraphics[height=.4in]{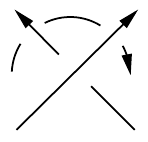}} \,\right> &=& -A^2 \left<\, \raisebox{-.4cm}{\includegraphics[height=.4in]{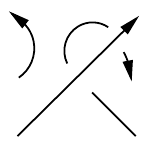}}\, \right>  -A^4 \left< \,\raisebox{-.4cm}{\includegraphics[height=.4in]{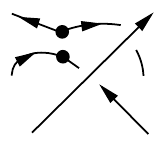}}\, \right> \\
&\stackrel{R_2}{=}& -A^2 \left< \,\raisebox{-.4cm}{\includegraphics[height=.4in]{pr34}}\, \right> -A^4 \left<\, \raisebox{-.4cm}{\includegraphics[height=.4in]{pr37}}\, \right>.
\end{eqnarray*}
The diagrams associated with the weight $-A^2$ have the same evaluations, since the diagrams differ by the move $R_2$. Thus, it remains to show the following:
\begin{eqnarray}\label{R3-proof}
\left<\, \raisebox{-.4cm}{\includegraphics[height=.4in]{pr33}} \,\right> =  \left< \,\raisebox{-.4cm}{\includegraphics[height=.4in]{pr37}}\, \right> .
\end{eqnarray}
Applying the defining skein relations for the crossings in the left hand side diagram of ~\eqref{R3-proof} leads to:
\begin{eqnarray*}
\left<\, \raisebox{-.4cm}{\includegraphics[height=.4in]{pr33}}\, \right> &=& -A^{-2} \left< \,\raisebox{-.4cm}{\includegraphics[height=.4in]{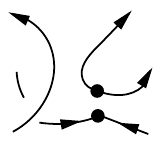}} \,\right> -A^{-4} \left<\, \raisebox{-.4cm}{\includegraphics[height=.4in]{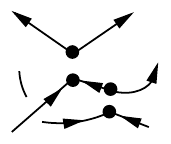}}\, \right> \\
&=& -A^{-2} \left( -A^2 \left<\, \raisebox{-.4cm}{\includegraphics[height=.4in]{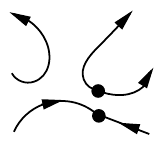}}\, \right> -A^4  \left< \,\raisebox{-.4cm}{\includegraphics[height=.4in]{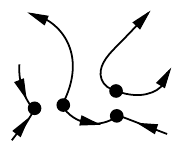}}\, \right> \right) \\
&&- A^{-4} \left( -A^2\left<\, \raisebox{-.4cm}{\includegraphics[height=.4in]{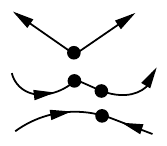}} \,\right> -A^4 \left< \,\raisebox{-.4cm}{\includegraphics[height=.4in]{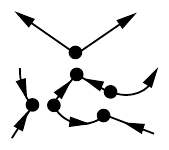}} \,\right> \right) \\
&=& \left<\, \raisebox{-.4cm}{\includegraphics[height=.4in]{pr310}} \,\right> + A^2 \left<\, \raisebox{-.4cm}{\includegraphics[height=.4in]{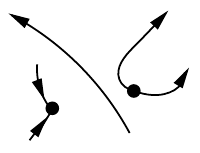}}\, \right> +A^{-2} \left<\, \raisebox{-.4cm}{\includegraphics[height=.4in]{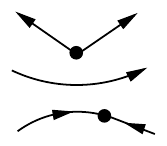}}\, \right> + \left<\, \raisebox{-.4cm}{\includegraphics[height=.4in]{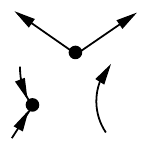}} \,\right>.
\end{eqnarray*}
Similarly, for the right hand side diagram of ~\eqref{R3-proof}, we obtain:
\begin{eqnarray*}
\left<\, \raisebox{-.4cm}{\includegraphics[height=.4in]{pr37}}\, \right> &=& -A^{-2} \left<\, \raisebox{-.4cm}{\includegraphics[height=.4in]{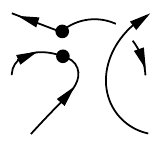}} \,\right> -A^{-4} \left< \,\raisebox{-.4cm}{\includegraphics[height=.4in]{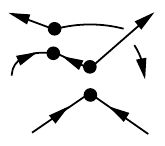}}\, \right> \\
&=& -A^{-2} \left( -A^2 \left<\, \raisebox{-.4cm}{\includegraphics[height=.4in]{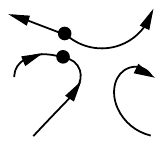}}\, \right> -A^4  \left<\, \raisebox{-.4cm}{\includegraphics[height=.4in]{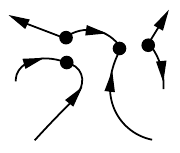}}\, \right> \right) \\
&&- A^{-4} \left( -A^2\left< \,\raisebox{-.4cm}{\includegraphics[height=.4in]{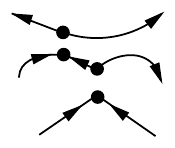}} \,\right> -A^4 \left< \,\raisebox{-.4cm}{\includegraphics[height=.4in]{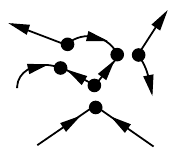}}\, \right> \right) \\
&=& \left< \,\raisebox{-.4cm}{\includegraphics[height=.4in]{pr319}} \,\right> + A^2 \left<\, \raisebox{-.4cm}{\includegraphics[height=.4in]{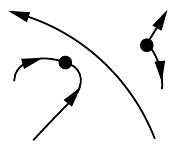}} \,\right> +A^{-2} \left< \,\raisebox{-.4cm}{\includegraphics[height=.4in]{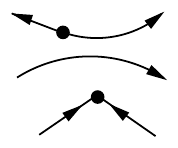}} \,\right> + \left<\, \raisebox{-.4cm}{\includegraphics[height=.4in]{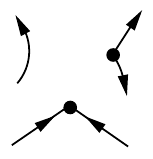}} \,\right>.
\end{eqnarray*}
Using planar isotopy, the diagrams with weights $A^2$ and, respectively, $A^{-2}$ on both sides of the desired identity are the same. Moreover, the first (or the last) diagram associated to the left hand side of the identity is the same as the last (or the first) diagram resulting from the right hand side of the identity. Thus, the identity~\eqref{R3-proof} holds, and therefore
\[
\left<\, \raisebox{-.4cm}{\includegraphics[height=.4in]{pr31}}\, \right> =  \left<\, \raisebox{-.4cm}{\includegraphics[height=.4in]{pr35}}\, \right> .
\]

The invariance under the other seven oriented versions of the Reidemeister move $R_3$ can be verified similarly (one can also use a generating set of Reidemeister moves argument).

To prove the invariance of $\left < \, \cdot \, \right >$ under the move $RS_1$ for virtual singular link diagrams, we use the skein relation \eqref{eq:sing-crossing} and that the polynomial is invariant under the Reidemeister move $R_3$, as we explain below:
\[
 \left <\, \raisebox{-.4cm}{\includegraphics[height=.4in]{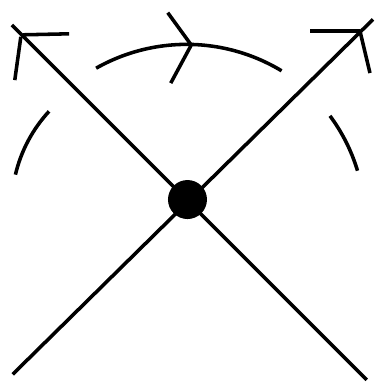}} \,\right > \, = \, \left < \,\raisebox{-.4cm}{\includegraphics[height=.4in]{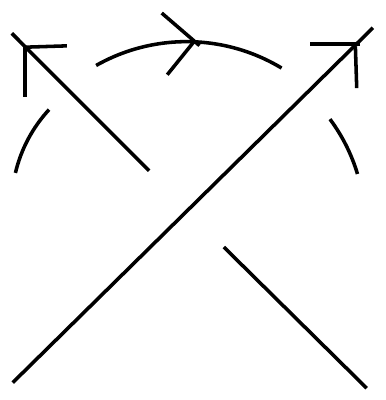}} \, \right >  \,+ \, \left < \,\raisebox{-.4cm}{\includegraphics[height=.4in]{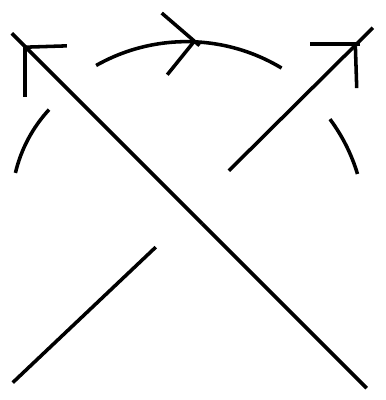}} \, \right >
\, =\,  \left < \,\raisebox{-.4cm}{\includegraphics[height=.4in]{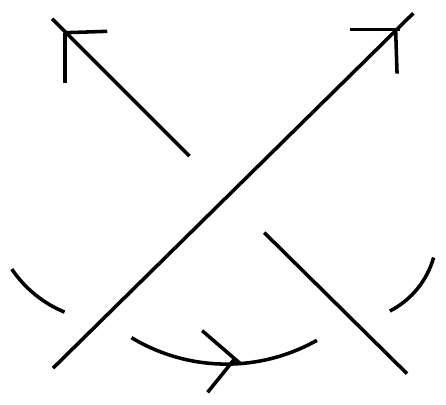}} \, \right > \, + \, \left < \,\raisebox{-.4cm}{\includegraphics[height=.4in]{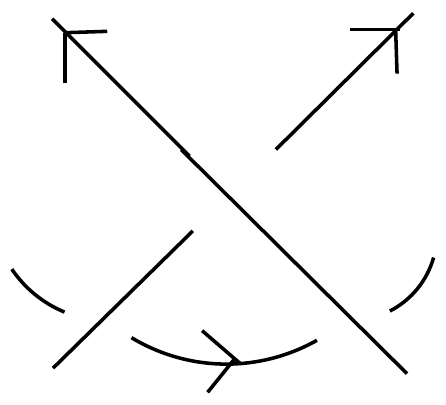}}\ \, \right >
\,=\,  \left < \, \raisebox{-.4cm}{\includegraphics[height=.4in]{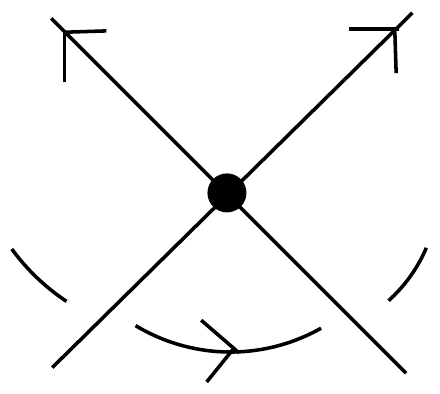}} \right >.  
\]

A similar approach is used to show the invariance under the move $RS_2$. First, we apply the skein relation~\eqref{eq:sing-crossing} to the singular crossing in the diagram on the left hand side of the move. We then use that the polynomial is invariant under the Reidemeister move $R_2$ to pull apart the strands in the second diagram of the resulting equality, and then we employ again the skein relation~\eqref{eq:sing-crossing}, to obtain the evaluation of the diagram on the right hand side of the move, as desired:
\begin{eqnarray*}
\left < \,\raisebox{-.3cm}{\includegraphics[height=.3in]{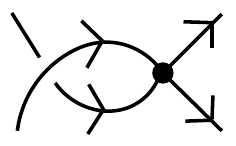}}\, \right >  &=& \left < \,\raisebox{-.35cm}{\includegraphics[height=.35in]{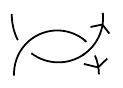}} \, \right > +  \left < \, \raisebox{-.35cm}{\includegraphics[height=.35in]{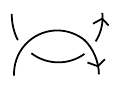}} \, \right > \\
& =& \left < \, \raisebox{-.35cm}{\includegraphics[height=.35in]{of2}} \,\right > +  \left < \,\raisebox{-.35cm}{\includegraphics[height=.35in]{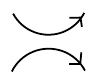}} \, \right >   \\ 
 &=& \left < \, \raisebox{-.35cm}{\includegraphics[height=.35in]{of2}} \, \right > +  \left < \,\raisebox{-.35cm}{\includegraphics[height=.35in]{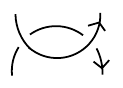}} \,\right > 
= \left < \, \raisebox{-.3cm}{\includegraphics[height=.3in]{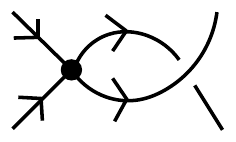}} \,\right >.  
\end{eqnarray*}

The invariance under the move $V_{3-c}$ follows from the defining skein relations for classical crossings and planar isotopy, as exemplified below:
\begin{eqnarray*}
\left < \, \raisebox{-.4cm}{\includegraphics[height=.4in]{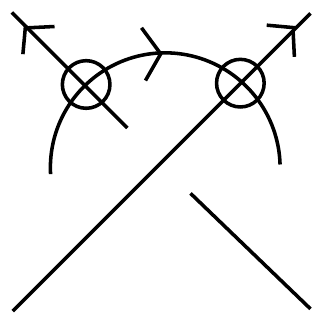}} \,\right > &=& -A^{-2} \left < \,\raisebox{-.4cm}{\includegraphics[height=.4in]{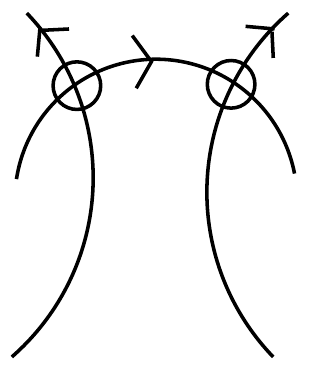}} \,\right > - A^{-4} \left < \,\raisebox{-.4cm}{\includegraphics[height=.4in]{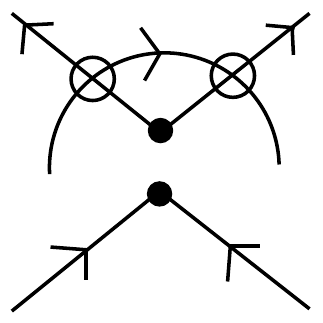}} \, \right > \\
 &=&  -A^{-2} \left < \, \raisebox{-.4cm}{\includegraphics[height=.4in]{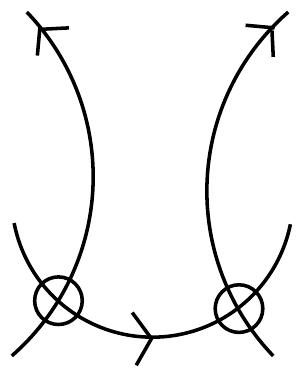}}\, \right > - A^{-4} \left < \,\raisebox{-.4cm}{\includegraphics[height=.4in]{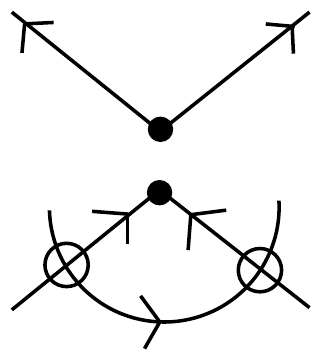}} \,\right > 
= \left <\, \raisebox{-.4cm}{\includegraphics[height=.4in]{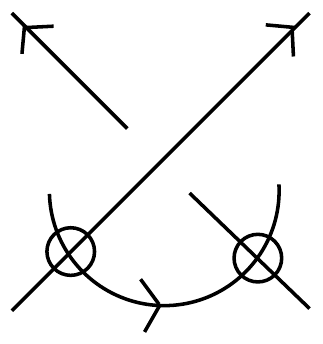}}\, \right >.
\end{eqnarray*}
Here, we used that the diagrams receiving the weight $-A^{-4}$ have the same evaluations, as the number of loops in the larger diagrams (which are identical except in the neighborhood shown) remain the same when the horizontal arc crosses virtually the top or the bottom two edges of the disoriented resolution. 

Lastly, the invariance of $\left < \, \cdot\, \right >$ under the move $V_{3-s}$ follows from the skein relation~\eqref{eq:sing-crossing} for the singular crossing and the now known invariance of $\left < \, \cdot \, \right >$ under the move $V_{3-c}$:
\[
 \left < \, \raisebox{-.4cm}{\includegraphics[height=.4in]{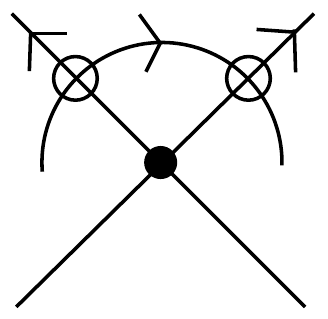}} \,\right >  = \left < \,\raisebox{-.4cm}{\includegraphics[height=.4in]{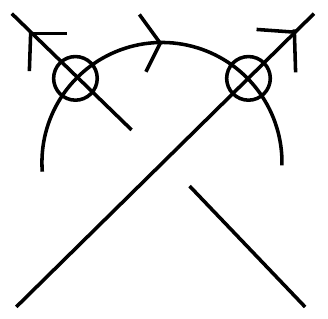}} \, \right > +  \left < \, \raisebox{-.4cm}{\includegraphics[height=.4in]{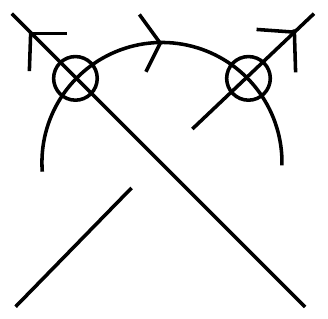}} \, \right >
 = \left < \, \raisebox{-.4cm}{\includegraphics[height=.4in]{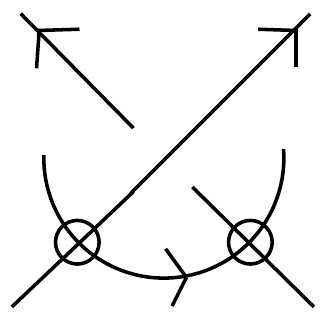}} \, \right > +  \left < \,\raisebox{-.4cm}{\includegraphics[height=.4in]{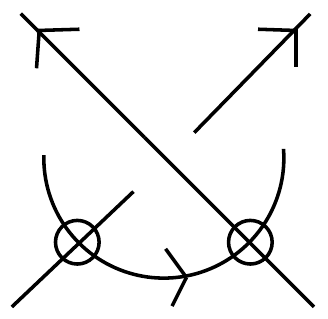}} \, \right >  
= \left < \, \raisebox{-.4cm}{\includegraphics[height=.4in]{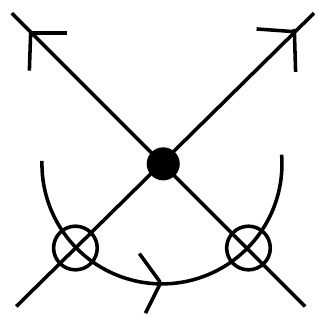}} \, \right > . 
\]
This completes the proof.
\end{proof}

With Theorem~\ref{thm:inv} at hand, we can now define the \textit{Kauffman-Jones polynomial of a virtual singular link $L$} as follows:
\[ \left <\,L\,\right > : = \left <D \right>, \]
where $D$ is any diagram representing $L$.

\begin{remark}
If $L$ is an oriented virtual link (that is, $L$ contains no singular crossings), our polynomial $\left < \, L \, \right >$ is the same as the polynomial $f_L(A)$ introduced in L. Kauffman's work~\cite{Ka2}. Moreover, if $L$ is an oriented classical link, then $\left < \, L \, \right >$ is the polynomial $f[L] (A)$ constructed by L. Kauffman in~\cite{Ka1}, which provides a state model for the Jones polynomial~\cite{Jo}.
\end{remark}

\begin{example}
We evaluate the Kauffman-Jones polynomial of an oriented figure-eight knot containing one classical crossing, one singular crossing, and two virtual crossings. We begin by applying the defining skein relation for the singular crossing followed by that for the classical negative crossing, as shown below:
\begin{eqnarray*}
 \left < \, \raisebox{-.4cm}{\includegraphics[height=.4 in]{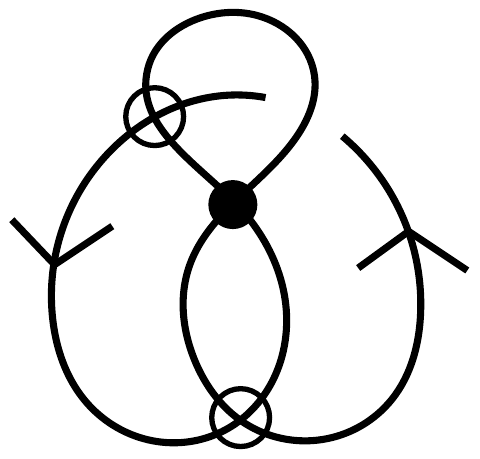}}\, \right > &=& (-A^2-A^{-2}) \left < \, \raisebox{-.4cm}{\includegraphics[height=.4 in]{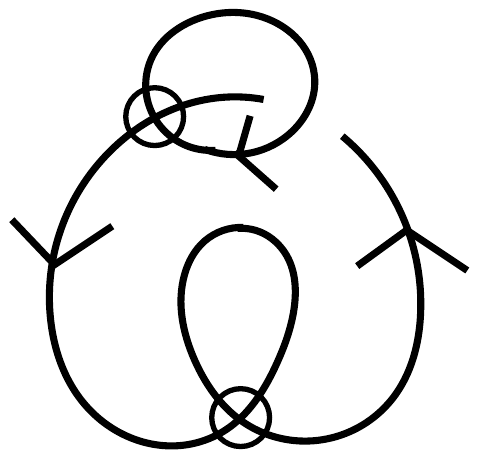}}\, \right > +(-A^4-A^{-4}) \left < \, \raisebox{-.4cm}{\includegraphics[height=.4 in]{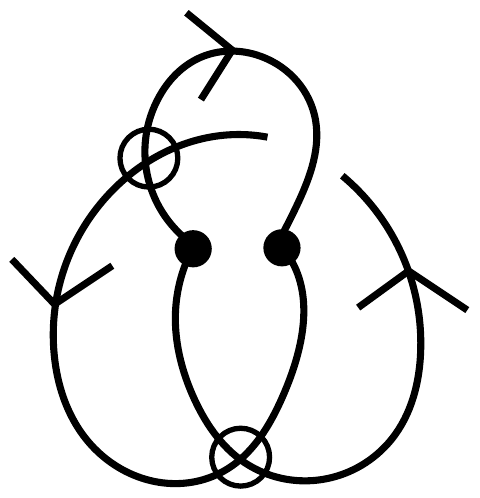}}\, \right > \\
&=& (-A^2-A^{-2}) \left[  -A^2 \left < \, \raisebox{-.4cm}{\includegraphics[height=.4 in]{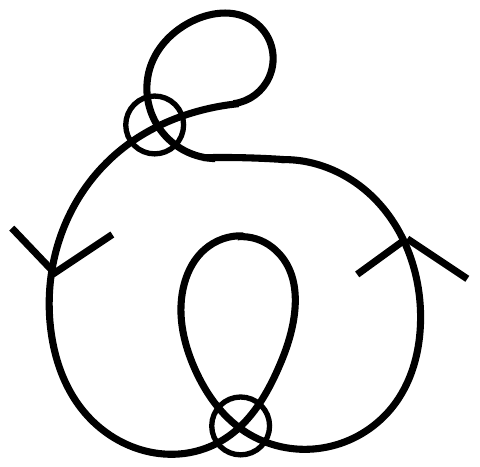}}\, \right > -A^4 \left < \, \raisebox{-.4cm}{\includegraphics[height=.4 in]{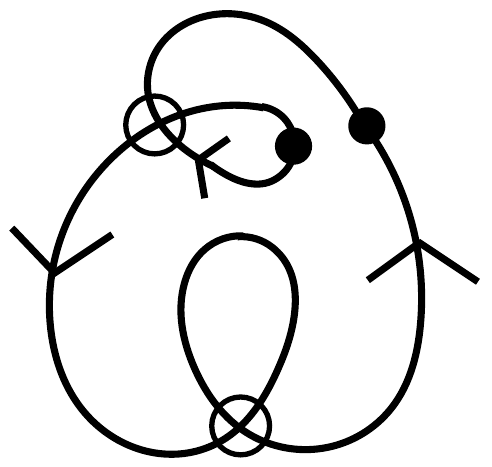}}\, \right > \right] \\
&& + (-A^4-A^{-4}) \left[  -A^2 \left < \, \raisebox{-.4cm}{\includegraphics[height=.4 in]{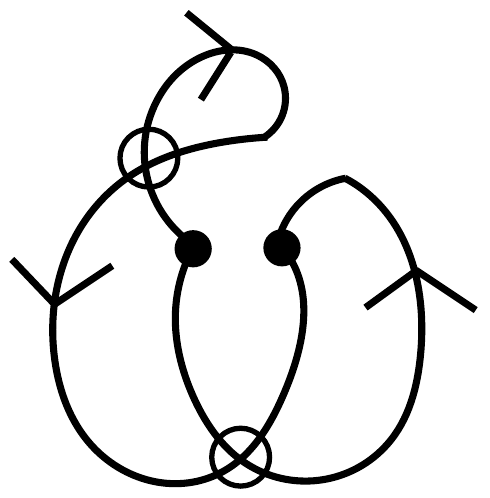}}\, \right > -A^4 \left < \, \raisebox{-.4cm}{\includegraphics[height=.4 in]{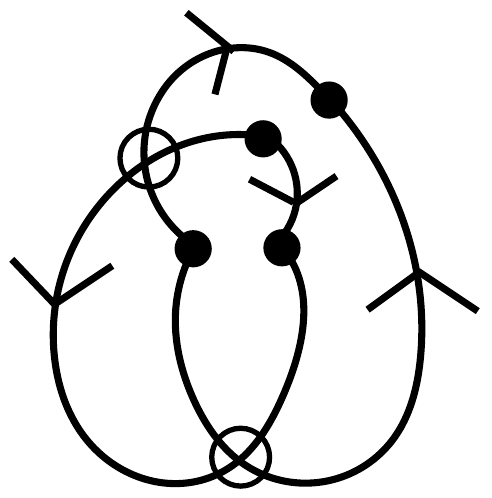}}\, \right > \right] \\
&=& (-A^2-A^{-2})[-A^2(-A^2-A^{-2}) -A^4(-A^2-A^{-2})] \\
&&+ (-A^4-A^{-4})[ -A^2(-A^2-A^{-2}) -A^4(-A^2-A^{-2})^2 ] \\
&=& (-A^2-A^{-2})(A^4+1+A^6+A^2) + (-A^4-A^{-4})(-A^8-A^4) \\
&=& A^{12}-A^6-A^4-2A^2-A^{-2}.
\end{eqnarray*}
\end{example}

We end this section by giving a state sum formula for the Kauffman-Jones polynomial of a virtual singular link. Before we do so, we need to go over some new notation. 

Consider a virtual Kauffman-Jones state $S$ obtained from a virtual singular link diagram $D$. Let $a(S)$ denote the number of negative classical crossings in $D$ that received an oriented resolution to obtain the state $S$ minus the number of positive classical crossings that received an oriented resolution to obtain $S$. Similarly, let $b(S)$ be the number of negative classical crossings in $D$ that received a disoriented resolution to obtain $S$ minus the number of positive crossings that received a disoriented resolution to obtain the state $S$. In addition, let $\alpha(S)$ denote the number of singular crossings in $D$ that received an oriented resolution, and let $\beta(S)$ denote the number of singular crossings that received a disoriented resolution to result in the state $S$. Finally, denote by $|| S ||$ the number of components in $S$ (immersed closed curves with self intersections represented as virtual crossings), and let $c(D)$ represent the number of classical crossings in $D$. 

With these definitions at hand together with the skein relations given in Figures~\ref{fig:extended-jones} and~\ref{fig:graph-skein-rel} (which define our polynomial), we obtain the following state sum formula for the Kauffman-Jones polynomial of a virtual singular link diagram $D$: 
\begin{eqnarray} \label{state_sum}
\hspace{1cm}  \left < \, D \, \right >  =  (-1)^{c(D)}\sum_{S} A^{2a(S)+4b(S)}(-A^2-A^{-2})^{\alpha(S)+|| S ||}\,(-A^4-A^{-4})^{\beta(S)}, 
\end{eqnarray} 
where the sum runs over all states $S$ associated with $D$.


\section{Another approach to the Kauffman-Jones polynomial}\label{sec:second-poly}

In this section, we take another approach to evaluating the Kauffman-Jones polynomial of a virtual singular link, which is motivated  by the work of N. Kamada and Y. Miyazawa in \cite{KM}. With this new perspective, we show that the Kauffman-Jones polynomial of a virtual singular link decomposes non-trivially into two parts, one belonging to $\mathbb{Z}[A^4,A^{-4}]$ and the other to $\mathbb{Z}[A^4,A^{-4}] \cdot A^2$. 

Given a virtual singular link diagram $D$, we still resolve the singular and classical crossings in $D$ using the skein relations given in Figure~\ref{fig:extended-jones}, but we use a slightly different method for evaluating the resulting Kauffman-Jones states (recall that these are purely virtual magnetic graphs).

Consider a purely virtual magnetic graph $S$. Let $E(S)$ denote the set of edges of $S$, where an edge is an arc between two adjacent bivalent vertices in $S$. An immersed closed curve with no bivalent vertices is considered to be an edge. 

\begin{definition}
A \textit{weight map} of a purely virtual magnetic graph $S$ is a function $\tau: E(S) \rightarrow \{1,-1\}$ such that adjacent edges $e$ and $e^\prime$ are assigned different values, that is, $\tau(e) \neq \tau(e^\prime)$. The assigned value to an edge $e$ is called the \textit{weight of $e$ with respect to $\tau$}. Given a weight map $\tau$ and a virtual crossing $v$ in $S$, the \textit{parity of $v$ with respect to $\tau$}, denoted $i_{\tau}(v)$, is defined as follows:
\[ i_{\tau}(v) = \tau(e) \tau(e^\prime), \]
where $e$ and $e^\prime$ are the two edges meeting at $v$. 
The \textit{parity of $S$} (with respect to $\tau$), denoted by $i_\tau(S)$, is defined as the product of the parities of all virtual crossings in $S$. That is,
\[ i_\tau(S) = \prod_{v\in S}    i_{\tau}(v).\]
\end{definition}

We remark that an immersed curve $S$ without bivalent vertices has parity 1.

As we now show, the parity of a purely virtual magnetic graph $S$ is independent of the choice of the weight map $\tau$, and consequently, we will use the notation $i(S)$ instead of $i_{\tau}(S)$. The proof of the following statement was given in~\cite{KM}, and we provide it here in order to have a self contained paper.

\begin{lemma}
Let $S$ be a purely virtual magnetic graph. Then, the parity of $S$ is independent on the choice of the weight map.
\end{lemma}

\begin{proof} 
Let $\tau$ and $\omega$ be two weight maps of $S$. By definition, the only possible values assigned to the edges of $S$ are $1$ or $-1$. We consider a partition of $E(S)$ into two subsets. Let $E_1(S)$ be the set of all edges $e \in E(S)$ such that $\tau(e)=\omega(e)$, and $E_2(S)$ be the set of all edges such that $\tau(e) \neq \omega(e)$. Note that $\tau(e) \neq \omega(e)$ is equivalent to $\tau(e)= - \omega(e)$. Consider the components of $S$ (that is, immersed closed curves in $S$). All of the edges of a component of $S$ will either belong to $E_1(S)$ or $E_2(S)$. To see this, consider an edge $e$ of a component of $S$. Suppose that $e \in E_1(S)$, so $\tau(e)=\omega(e)$, and consider an adjacent edge $e^\prime$. If $\tau(e)=\omega(e)=a$, then by definition of a purely virtual magnetic graph, we have $\tau(e^\prime)=\omega(e^\prime)=-a$, and thus $e^\prime \in E_1(S)$. We repeat this argument as we move through the edges of the component of $S$, to conclude that all of its edges belong to $E_1(S)$. The same argument works for the case when $e \in E_2(S)$. Thus all of the edges of a component will either belong to $E_1(S)$ or $E_2(S)$. 

Let $S_i$ denote the components of $S$ whose edges belong to $E_i(S)$, for $i=1,2$. Consider a virtual crossing $v$ of $S$. The parity of $v$ with respect to $\tau$ will be different from the parity of $v$ with respect to $\omega$ if and only if $v$ is a crossing between an edge from $S_1$ and an edge from $S_2$. However, two components of $S$ will always intersect in an even number of virtual crossings, and therefore, the product of the parities of the virtual crossings that could be different (with respect to $\tau$ and $\omega$) will be equal to $1$, since there is an even number of such virtual crossings. Hence, $i_\tau(S) = i_\omega(S)$.
\end{proof}

\begin{definition}
An \textit{enhanced purely virtual magnetic graph} is a purely virtual magnetic graph $S$ together with a weight map of $S$.
\end{definition}

We are now ready to define the new method for evaluating an enhanced purely virtual magnetic graph. As before, we denote by $|| S ||$ the number of components in $S$.

\begin{definition}\label{def:enhanced-mg}
Given an enhanced purely magnetic graph $S$ with parity $i(S)$, the \textit{enhanced evaluation} of $S$, denoted $R(S)$, is given by:
\begin{eqnarray*}\label{eq:enhanced-eval}
 R(S) =  (-A^2-A^{-2})^{|| S ||}h^{\frac{1-i(S)}{2}}
 \end{eqnarray*}
 where $A, A^{-1}$ and $h$ are commuting parameters.
\end{definition}

\begin{remark}
If we substitute $h =1$ in the definition of the enhanced evaluation of $S$, we obtain the evaluation of $S$, as defined before.
\end{remark}

\begin{example} Consider the following enhanced purely magnetic graph $S$ given below:

\[\raisebox{30pt}{S =} \hspace{0.5cm} \includegraphics[height=1.2in]{VMGstate_ex} \put(-5, 5){\fontsize{9}{11}$1$} \put(-40, -2){\fontsize{9}{11}$-1$} \put(-50, 5){\fontsize{9}{11}$1$} \put(-70, 1){\fontsize{9}{11}$1$} \put(-90, 45){\fontsize{9}{11}$-1$}  \put(-81, 75){\fontsize{9}{11}$-1$}  \put(-22, 60){\fontsize{9}{11}$1$}
 \]
There are four components in $S$, and thus $||S||=4$. The virtual crossing in the top component has parity 1, and both of the virtual crossings in the bottom right component have parity $-1$. Thus $i(S)=(1)(-1)(-1)=1$ and the enhanced evaluation of this state is $R(S) = (-A^2-A^{-2})^4 \, h^{\frac{1-1}{2}}=(-A^2-A^{-2})^4.$
\end{example}

Consider a virtual singular link diagram $D$ and its associated Kauffman-Jones states together with a weight map for each of the states, which now are \textit{enhanced Kauffman-Jones states}. We employ the symbols used in the state sum formula given in Equation~\eqref{state_sum}, to define a new polynomial associated to $D$.

\begin{definition}\label{def:new-poly}
Given a virtual singular link diagram $D$, we define a polynomial $R(D) \in \mathbb{Z}[A,A^{-1},h]$ as the following sum, taken over all enhanced Kauffman-Jones states $S$ associated with $D$:
\[
R(D)=   (-1)^{c(D)} \sum_S A^{2a(S)+4b(S)}(-A^2-A^{-2})^{\alpha(S)+|| S ||}(-A^4-A^{-4})^{\beta(S)}h^{\frac{1-i(S)}{2}}.
\]
\end{definition}

\begin{remark}
Comparing the defining formula for $R(D)$ with the state sum formula for the Kauffman-Jones polynomial given in Equation~\eqref{state_sum}, we see that $R(D)_{|h = 1}  = \left < \,D\, \right >$. The role of the indeterminate $h$ in our construction will become clear in Section~\ref{sec:split}, where we address a splitting property of the polynomial $R(D)$, and consequently, of the polynomial $\left < \,D\, \right >$.
\end{remark}

\begin{remark}
Given a virtual singular link diagram $D$ and $S$ an enhanced  Kauffman-Jones state of $D$, denote by $\OR(S)$ the \textit{weighted state contribution} of $S$ to $R(D)$:
\[ \OR(S):= A^{2a(S)+4b(S)}(-A^2-A^{-2})^{\alpha(S)+|| S ||}(-A^4-A^{-4})^{\beta(S)}h^{\frac{1-i(S)}{2}} .\]
Then the polynomial $R(D)$ can be rewritten as:
\[ R(D) =  (-1)^{c(D)} \sum_S \OR(S). \]
\end{remark}

\begin{proposition}\label{prop:skeinR}
Given a virtual singular link diagram $D$, the polynomial $R(D)$ can be computed iteratively via the skein relations,

 \[ R \left( \, \raisebox{-10pt}{\includegraphics[height=.35in]{opos}}\, \right) = -A^{-2} R \left ( \,  \raisebox{-10pt}{\includegraphics[height=.35in]{oriented}} \, \right) -  A^{-4} R \left( \,  \raisebox{-10pt}{\includegraphics[height=.35in]{unoriented}} \, \right) \] 

 \[ R \left ( \, \raisebox{-10pt}{\includegraphics[height=.35in]{oneg}}\, \right )= -A^{2} R \left ( \, \raisebox{-10pt}{\includegraphics[height=.35in]{oriented}} \, \right) -  A^{4} R \left ( \, \raisebox{-10pt}{\includegraphics[height=.35in]{unoriented}} \, \right) \] 

\[ 
R \left ( \, \raisebox{-10pt}{\includegraphics[height=.35in]{osingular}} \,\right ) = (-A^{2}-A^{-2}) R \left ( \, \raisebox{-10pt}{\includegraphics[height=.35in]{oriented}}\, \right ) + (-A^4-A^{-4}) R\left ( \, \raisebox{-10pt}{\includegraphics[height=.35in]{unoriented}}\, \right )
\]
followed by the application of Definition~\ref{def:enhanced-mg}, to evaluate the resulting enhanced Kauffman-Jones states corresponding to $D$.
\end{proposition}

\begin{proof}
The statement follows easily by comparing the polynomials $\left < \, D \, \right>$ and $R(D)$.
\end{proof}

\begin{proposition} \label{prop:enhanced-graph-rel}
The enhanced evaluation $R$ for enhanced purely virtual magnetic graphs satisfies the following skein relations:
\[
R \left ( \raisebox{-.2cm}{\includegraphics[height=.2in]{rsmooth}} \right ) = R \left ( \raisebox{-.15cm}{\includegraphics[height=.2in]{lcurve}} \right )  \hspace{1cm} R \left ( \reflectbox{\raisebox{-.2cm}{\includegraphics[height=.2in]{rsmooth}}} \right ) = R \left (\reflectbox{ \raisebox{-.15cm}{\includegraphics[height=.2in]{lcurve}} }\right )
\]
\[ R \left ( \raisebox{-.15cm}{\includegraphics[height=.2in]{pspin}} \right ) = -A^2-A^{-2} = R \left ( \raisebox{-.2cm}{\includegraphics[height=.2in]{nspin}} \right )
\]
\[
R  \left ( \Gamma \cup \raisebox{-.15cm}{\includegraphics[height=.2in]{pspin}} \right ) = (-A^2 -A^{-2} ) R \left ( \Gamma \right ) = R \left ( \Gamma \cup \raisebox{-.15cm}{\includegraphics[height=.2in]{nspin}} \right ).
\]
Moreover, the parity for enhanced purely virtual magnetic graphs satisfies the following skein relations:
 \[
i \left ( \, \raisebox{-.3cm}{\includegraphics[height=.3in]{vv11}} \, \right ) =  - i \left ( \, \raisebox{-.3cm}{\includegraphics[height=.3in]{vv12}} \, \right )  \hspace{1cm} i \left ( \, \raisebox{-.3cm}{\includegraphics[height=.3in]{vv21}} \, \right ) = - i \left (  \, \raisebox{-.3cm}{\includegraphics[height=.3in]{vv22}} \, \right ).
\]
\end{proposition}

\begin{proof}
These skein relations follow at once from Definition~\ref{def:enhanced-mg} and the definition of the parity of enhanced purely virtual magnetic graphs.
\end{proof}

\begin{theorem}\label{thm:invR}
The polynomial $R(\, \cdot \,)$ is an ambient isotopy invariant for virtual singular links. 
\end{theorem}

\begin{proof}
We have that $R(D)_{|h = 1} = \left < \,D\, \right >$ and that the two polynomials $\left < \, D \, \right>$ and $R(D)$ differ only by how the Kauffman-Jones states are evaluated.  In particular, the polynomial $R$ satisfies the same skein relations (for the classical and virtual crossings) as the polynomial $\left < \,D\, \right >$. Since in the proof of Theorem~\ref{thm:inv} we used these skein relations to show that $\left < \, D \, \right>$ is invariant under the extended Reidemeister moves, it is clear that the polynomial $R(\, \cdot \,)$ is invariant under the moves that involve only classical and singular crossings. Therefore, it remains to show that $R(\, \cdot \,)$ is unchanged under the extended Reidemeister moves that involve virtual crossings, namely the last five moves given in Figure~\ref{fig:isotopies}.

\textit{Invariance under the move $V_1$.} Consider two virtual singular link diagrams $D$ and $D^{\prime}$ that differ by a virtual move $V_1$:
\[ D =  \raisebox{-.4cm}{\includegraphics[height=.4In]{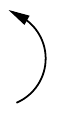}} \hspace{1cm} \text{and}  \hspace{1cm}  D^{\prime} = \raisebox{-.4cm}{\includegraphics[height=.4in]{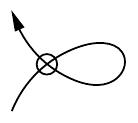}}\]
Each state $S^{\prime}$ of $D^ {\prime}$ can be obtained from a state $S$ of $D$ by applying a virtual twist, where the twist is applied in the same small neighborhood of the point where the twist was introduced on $D$. Since the parity of the virtual crossing involved in the twist is always $1$, regardless of the weight map of the state, it follows that $R(S) = R(S^ {\prime})$ for every corresponding pair of states $S$ and $S^{\prime}$, and therefore $R(D) = R(D^ {\prime})$. 

\textit{Invariance under the move $V_2$.} Consider two virtual singular link diagrams $D$ and $D^{\prime}$ that are identical except near a point where they differ as shown:
\[ D =  \raisebox{-.4cm}{\includegraphics[height=.4In]{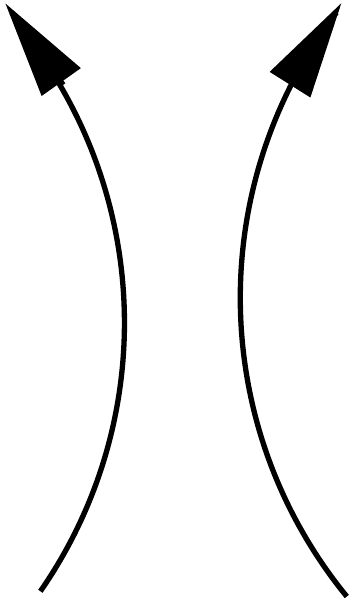}} \hspace{1cm} \text{and}  \hspace{1cm}  D^{\prime} = \raisebox{-.4cm}{\includegraphics[height=.4in]{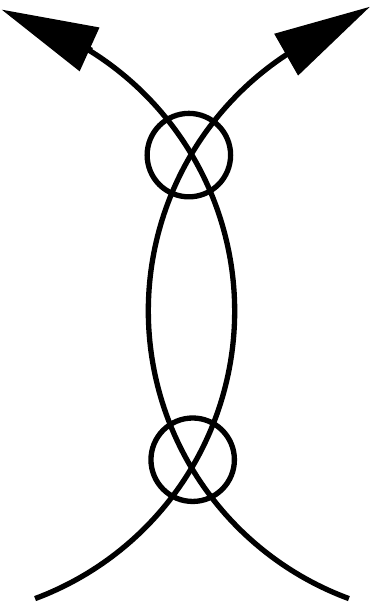}}\]
The states $S$ of $D$ are in one-to-one correspondence with the states $S^{\prime}$ of $D^ {\prime}$, and each corresponding pair $S$ and $S^{\prime}$ differ by the virtual move $V_2$. The parities of the additional two virtual crossings in $S^{\prime}$ are the same, regardless of the weight map of the state, and thus the overall parities of each corresponding pairs $S$ and $S^{\prime}$ are the same. Hence, $R(D) = R(D^{\prime})$.

\textit{Invariance under the move $V_{3-v}$.} Consider again two diagrams that are identical except near a point where they differ as shown below:
\[
D = \, \raisebox{-.4cm}{\includegraphics[height=.4in]{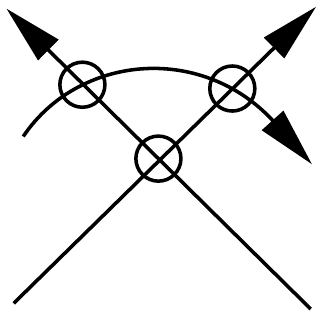}} \hspace{1cm} \text{and} \hspace{1cm}  D^{\prime} =\raisebox{-.4cm}{\includegraphics[height=.4in]{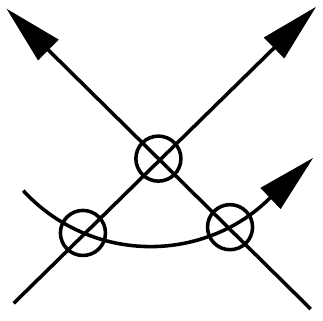}} 
\]
No virtual crossings are added or removed during the move, and the number of components in the resulting states of $D$ and $D^{\prime}$ is unaffected. As before, there is a one-to-one correspondence between the states $S$ of $D$ and $S^{\prime}$ of $D^ {\prime}$, respectively. Each corresponding pair $S$ and $S^{\prime}$ differ by the virtual move $V_{3-v}$.   We need to show that the parities of the three crossings involved in the move are the same when the move is applied, implying that the overall parities of the corresponding states $S$ and $S^{\prime}$ are the same. Given an assigned weight map of a state, the horizontal strand intersecting virtually the two diagonal strands will have the same weight whether it is above the central virtual crossing or below. Moreover, the move $V_{3-v}$ does not affect the weights of the diagonal strands meeting at the central virtual crossing. It follows that $i(S) = i(S^{\prime})$ for all corresponding pairs of states $S$ and $S^{\prime}$, and hence $R(D) = R(D^{\prime})$ in this case as well.

\textit{Invariance under the move $V_{3-r}$.} Let $D$ and $D^{\prime}$ be virtual singular link diagrams that differ only in a small neighborhood as shown below:
\[D = \raisebox{-.4cm}{\includegraphics[height=.4in]{ovc1}} \hspace{.2cm}   \hspace{1cm} D^{\prime} = \raisebox{-.5cm}{\includegraphics[height=.4in]{ovc6}} \]
We have that
\[R(D) =  -A^{-2} R \left( \,\raisebox{-.4cm}{\includegraphics[height=.4in]{ovc2}} \, \right) -A^{-4}  R \left( \,\raisebox{-.4cm}{\includegraphics[height=.4in]{ovc3}} \, \right) \]
and 
\[ R(D^{\prime}) = -A^{-2} R \left( \, \raisebox{-.4cm}{\includegraphics[height=.4in]{ovc4}}\, \right) -A^{-4} R \left(  \,\raisebox{-.4cm}{\includegraphics[height=.4in]{ovc5}} \,\right).  
\] 
The second diagrams in the above two skein relations contain the same number of components, and we need to verify that they have the same parity. Since these two diagrams are identical outside of the neighborhood shown, the only chance of having different parities depends on the parities of the virtual crossings shown. Consider in both of the diagrams the two adjacent edges intersecting virtually with the horizontal arc. By the definition of a weight map, the two adjacent edges will receive opposite weights, and therefore, no matter the weight of the horizontal strand, the two shown virtual crossings in either of the diagrams will have opposite parities. Hence, 
\[ i\left (\,\raisebox{-.4cm}{\includegraphics[height=.4in]{ovc3}}  \, \right) = i\left (\,\raisebox{-.4cm}{\includegraphics[height=.4in]{ovc5}}  \, \right) \Longrightarrow   R\left (\,\raisebox{-.4cm}{\includegraphics[height=.4in]{ovc3}}  \, \right) = R\left (\,\raisebox{-.4cm}{\includegraphics[height=.4in]{ovc5}}  \, \right).  \]
It follows that $R(D) = R(D^{\prime})$. The case when the classical crossing involved in the move is negative follows in a similar fashion.

\textit{Invariance under the move $V_{3-s}$.} We note first that
\[ 
R \left ( \,\raisebox{-10pt}{\includegraphics[height=.35in]{osingular}} \,\right ) = R \left ( \, \raisebox{-10pt}{\includegraphics[height=.35in]{opos}} \,\right ) + R \left ( \,\raisebox{-10pt}{\includegraphics[height=.35in]{oneg}} \,\right ).
\]
Then, the invariance under the move $V_{3-s}$ follows as shown below, where we use that $R$ is invariant under the move $R_{3-c}$.

\begin{eqnarray*}
R \left ( \, \raisebox{-.4cm}{\includegraphics[height=.4in]{osc1}} \,\right )  &=& R \left ( \,\raisebox{-.4cm}{\includegraphics[height=.4in]{osc2}} \, \right ) + R \left ( \, \raisebox{-.4cm}{\includegraphics[height=.4in]{osc3}} \, \right )\\
 &=& R \left ( \, \raisebox{-.4cm}{\includegraphics[height=.4in]{osc4}} \, \right ) +  R \left ( \,\raisebox{-.4cm}{\includegraphics[height=.4in]{osc5}} \, \right )  \\
&=& R \left ( \, \raisebox{-.4cm}{\includegraphics[height=.4in]{osc6}} \, \right ). 
\end{eqnarray*}
This completes the proof.
\end{proof}

\begin{remark}\label{rem:splitting-def}
The definition of the polynomial $R(D)$ implies that
\begin{eqnarray*}
R(D) &=&  (-1)^{c(D)} \left ( \sum_{S, i(S) = -1} \OR(S) +  \sum_{S, i(S) = 1} \OR(S)  \right ) \\
&=& \phi (D) h + \psi(D),
\end{eqnarray*}
where 
\begin{eqnarray*}
\phi(D): &=& (-1)^{c(D)} \sum_{S,i(S) = -1 } A^{2a(S)+4b(S)}(-A^2-A^{-2})^{\alpha(S)+|| S ||}(-A^4-A^{-4})^{\beta(S)}\\
 \psi(D): &=& (-1)^{c(D)} \sum_{S,i(S) = 1 } A^{2a(S)+4b(S)}(-A^2-A^{-2})^{\alpha(S)+|| S ||}(-A^4-A^{-4})^{\beta(S)},
 \end{eqnarray*}
and $\phi(D), \psi(D) \in \mathbb{Z}[A^2, A^{-2}]$.
\end{remark}

\begin{proposition}
If $D$ and $D'$ are virtual singular link diagrams that differ by an extended Reidemeister move, then $\phi(D) = \phi(D^\prime)$ and $\psi(D) = \psi(D^\prime)$. That is, the polynomials $\phi(\, \cdot \,)$ and $\psi(\, \cdot \,)$ are invariants for virtual singular links.
\end{proposition}

\begin{proof}
It is clear that polynomials $\phi(\, \cdot \,)$ and $\psi(\, \cdot \,)$ satisfy the same skein relations for the classical and singular crossings as the polynomials $R(\, \cdot \,)$ and $\left<\, \cdot \, \right >$. Moreover,

\[ \phi \left ( \raisebox{-.15cm}{\includegraphics[height=.2in]{pspin}} \right ) = -A^2-A^{-2} = \phi \left ( \raisebox{-.2cm}{\includegraphics[height=.2in]{nspin}} \right )
\]
\[
\phi  \left ( \Gamma \cup \raisebox{-.15cm}{\includegraphics[height=.2in]{pspin}} \right ) = (-A^2 -A^{-2} ) \phi \left ( \Gamma \right ) = \phi \left ( \Gamma \cup \raisebox{-.15cm}{\includegraphics[height=.2in]{nspin}} \right )
\]
\[
\phi \left ( \raisebox{-.2cm}{\includegraphics[height=.2in]{rsmooth}} \right ) = \phi \left ( \raisebox{-.15cm}{\includegraphics[height=.2in]{lcurve}} \right )  \hspace{1cm} \phi \left ( \reflectbox{\raisebox{-.2cm}{\includegraphics[height=.2in]{rsmooth}}} \right ) = \phi \left (\reflectbox{ \raisebox{-.15cm}{\includegraphics[height=.2in]{lcurve}}} \right ),
\]
and similarly, \[ \psi \left ( \raisebox{-.15cm}{\includegraphics[height=.2in]{pspin}} \right ) = -A^2-A^{-2} = \psi \left ( \raisebox{-.2cm}{\includegraphics[height=.2in]{nspin}} \right )
\]
\[
\psi  \left ( \Gamma \cup \raisebox{-.15cm}{\includegraphics[height=.2in]{pspin}} \right ) = (-A^2 -A^{-2} ) \psi \left ( \Gamma \right ) = \psi \left ( \Gamma \cup \raisebox{-.15cm}{\includegraphics[height=.2in]{nspin}} \right )
\]
\[
\psi \left ( \raisebox{-.2cm}{\includegraphics[height=.2in]{rsmooth}} \right ) = \psi \left ( \raisebox{-.15cm}{\includegraphics[height=.2in]{lcurve}} \right )  \hspace{1cm} \psi \left ( \reflectbox{\raisebox{-.2cm}{\includegraphics[height=.2in]{rsmooth}}} \right ) = \psi \left ( \reflectbox{\raisebox{-.15cm}{\includegraphics[height=.2in]{lcurve}}} \right ).
\]

Then, the same proof as in Theorem~\ref{thm:inv} can be used to show that $\phi(\, \cdot \,)$ and $\psi(\, \cdot \,)$ are invariant under the classical Reidemeister moves $R_1, R_2$ and $R_3$, as well as the moves $RS_1$ and $RS_2$. Moreover, the proof of Theorem~\ref{thm:invR} implies that $\phi(\, \cdot \,)$ and $\psi(\, \cdot \,)$ are invariant under the moves $V_1, V_2, V_{3-v}, V_{3-r}$ and $V_{3-s}$.
\end{proof}

\begin{definition}
Given a virtual singular link $L$, define the polynomial $R(L) \in \mathbb{Z}[A^2, A^{-2}, h]$ by $R( L ) : = R( D),$
where $D$ is any diagram representing $L$.
\end{definition}

\begin{example}\label{ex:2}
In this example, we evaluate the polynomial $R(D)$ of the diagram $D$ given below:
\[ D = \raisebox{-.5cm}{\includegraphics[height=.5 in]{fe1}}\]
The enhanced Kauffman-Jones states associated with $D$ are given in the first column in Table~\ref{table-example} (with certain weight maps), while the third column contains the weighted state contributions $\OR(S)$  to $R(D)$, for each of the enhanced states $S$ of $D$. Recall that 
\[\OR(S) = A^{2a(S)+4b(S)}(-A^2-A^{-2})^{\alpha(S)+|| S ||}(-A^4-A^{-4})^{\beta(S)}h^{\frac{1-i(S)}{2}}.\]
Taking the sum of the weighted state contributions, we have:
\begin{eqnarray*}
 R \left(  \, \raisebox{-.5cm}{\includegraphics[height=.5 in]{fe1}}\, \right) 
 &=& (-1) ^{c(D)} \left(\, \OR(S_1) + \OR(S_2) + \OR(S_3) + \OR(S_4) \,\right)  \\
 &=& - \left(\, \OR(S_1) + \OR(S_2) + \OR(S_3) + \OR(S_4) \,\right)  \\
&=& A^{12}h-A^4h-A^6-2A^2-A^{-2}\\
&=& h(A^{12} -A^4) + (-A^4 -2 -A^{-4}) A^2.
\end{eqnarray*}

\begin{table}[ht]
\caption{The Kauffman-Jones states of $D$ and their contributions to $R(D)$}
\label{table-example}
  \begin{tabular}{| | c | c | c | |}
    \hline \hline
     $S_1 =  \raisebox{-.47cm}{\includegraphics[height=.45 in]{fe4}}$ 
     \put(-7, 7){\fontsize{8}{10}$1$}&
 $\begin{array}{cc}
 a (S_1)=1 & b (S_1)=0\\
\alpha (S_1)=1 & \beta (S_1)=0 \\
||S_1||=1 & i(S_1)=1
 \end{array} $  & $\OR(S_1) = A^2(-A^2-A^{-2})^2$ \\
    \hline 
 $S_2 =  \raisebox{-.52cm}{\includegraphics[height=.45 in]{fe5}}$
   \put(-35, -10){\fontsize{8}{10}$1$}   \put(-16, 15){\fontsize{8}{10}$-1$} &
 $\begin{array}{cc}
 a (S_2)=0 & b (S_2)=1\\
\alpha (S_2)=1 & \beta (S_2)=0\\
||S_2||=1 & i(S_2)= -1
 \end{array} $   & $\OR(S_2) = A^4(-A^2-A^{-2})^2h$  \\
    \hline
$S_3 =  \raisebox{-.47cm}{\includegraphics[height=.45 in]{fe6}}$
  \put(-3, -7){\fontsize{8}{10}$1$}  \put(-40, -7){\fontsize{8}{10}$-1$} &
 $\begin{array}{cc}
 a (S_3)=1 & b (S_3)=0\\
\alpha (S_3)=0 & \beta (S_3)=1\\
||S_3||=1 & i(S_3)= -1
 \end{array} $ & $\OR(S_3) = A^2(-A^2-A^{-2})(-A^4-A^{-4})h$  \\
    \hline 
   $S_4 =  \raisebox{-.5cm}{\includegraphics[height=.45 in]{fe7}}$
    \put(-3, -9){\fontsize{8}{10}$1$}  \put(-17, 15){\fontsize{8}{10}$-1$}
      \put(-40, -7){\fontsize{8}{10}$-1$}  \put(-12, 3){\fontsize{8}{10}$1$}&
 $\begin{array}{cc}
 a (S_4)=0 & b (S_4)=1\\
\alpha (S_4)=0 & \beta (S_4)=1\\
||S_4||=2 & i(S_3)=-1
 \end{array} $ & $\OR(S_4) = A^4(-A^2-A^{-2})^2(-A^4-A^{-4})h$  \\
\hline \hline
    \end{tabular}
  \end{table}
 \end{example}  


\section{The splitting property}\label{sec:split}
 
In Example~\ref{ex:2}, we noticed that $R(D)=\phi(D)h + \psi(D),$ where $\phi(D) \in \mathbb{Z}[A^4,A^{-4}]$ and $\psi(D) \in \mathbb{Z}[A^4,A^{-4}] \cdot A^{2}$. Inspired by the work in~\cite{KM}, we show now that the polynomial $R(\, \cdot \,)$ always has this property. The main result is as follows:

\begin{theorem} \label{thm:splitting}
Let $L$ be a virtual singular link with $k$ components. Then 
\[R(L)=\phi(L)h + \psi(L),\] where $\phi(L) \in \mathbb{Z}[A^4,A^{-4}] \cdot A^{2(k-1)}$ and $\psi(L) \in \mathbb{Z}[A^4,A^{-4}] \cdot A^{2k}$.
\end{theorem}

The proof of Theorem~\ref{thm:splitting} is similar in spirit to the proof of Theorem 3 in~\cite{KM}. 

Before we prove our theorem, we need to look at a few lemmas. For that, consider a weighted state contribution $\OR(S)$ to the polynomial $R(D)$, where $S$ is any enhanced Kauffman-Jones state associated with a diagram $D$, and note that the powers of $A$ in all of the monomials of $\OR(S)$ are congruent to each other modulo 4. Denote by $\text{max}_A(\OR(S))$ the maximum power of $A$ in $\OR(S)$. That is, 
\[ \text{max}_A(\OR(S) ): = 2a(S) + 4b(S) + 2(\alpha(S) + ||S||) + 4 \beta(S)  \]
where $a(S), b(S), \alpha(S), \beta(S)$ and $||S||$ were defined at the end of Section~\ref{sec:poly-def}.

\begin{lemma}\label{proof thm-lemma1}
Let $D$ be a connected virtual singular link diagram with $k$-components and denote by $S_0$ the enhanced Kauffman-Jones state of $D$ obtained by resolving each of the classical and singular crossings in $D$ using the oriented resolution. Then, 
\[\OR(S_0) \in \mathbb{Z}[A^4, A^{-4}] \cdot A^{2k}.\]
\end{lemma}

\begin{proof}
We first note that for the oriented state $S_0$, we have:
\[ a(S_0) = -w(D),\,\, b(S_0) = 0,\,\, \alpha(S_0) = s(D),\, \, \beta(S_0) = 0,  \]
where $w(D)$ is the \textit{writhe} of $D$ (that is, the sum of the signs of the classical crossings in $D$) and $s(D)$ is the number of singular crossings in $D$.

The state $S_0$ has no bivalent vertices, and we consider the weight map that assigns 1 to all loops in $S_0$. Therefore, $i(S_0) = 1$ and $\OR(S_0) \in \mathbb{Z}[A^2, A^{-2}]$.

By hypothesis, the original diagram $D$ is connected. Consider any two linked components of  $D$ and let $p$ be a classical crossing formed by the two components. If we resolve the crossing $p$ in the oriented fashion, we obtain a connected virtual singular link diagram with one less component. Therefore, it is clear that there exist $k-1$ classical crossings in $D$, such that by applying the oriented resolution to all of these crossings results in a virtual singular knot diagram, call it $D_1$. That is, $D_1$ is a one-component virtual singular link diagram satisfying
\[ c(D_1) = c(D) - (k-1) \,\, \text{and} \,\, s(D_1) = s(D).  \]

Choose a base-point on $D_1$ and consider the walk along $D_1$ that starts at the base-point and proceeds according to the diagram's orientation. When arriving at a singular crossing, the walk continues straight ahead. The walk ends when returning to the base-point for the first time. Label the classical and singular crossings in $D_1$ and list them in the order they are reached via the walk.
If there are two classical and/or singular crossings, $p_a$ and $p_b$, which alternate in the list produced by the walk, as shown below:
\[ \cdots p_a, \cdots, p_b, \cdots, p_a, \cdots p_b ,\cdots ,  \]
resolve these two crossings using the oriented resolution. This process results in a (connected) virtual singular knot diagram, which we denote by $D_2$, as illustrated below.

\[ 
\raisebox{-1.2cm}{\includegraphics[height=0.9 in]{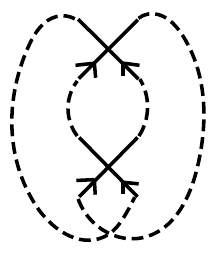}} \put(-31,-41){$D_1$} \put(-42,17){$p_b$} \put(-42,-12){$p_a$} \Longrightarrow \raisebox{-1.2cm}{\includegraphics[height=.9 in]{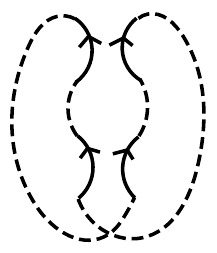}} \put(-31,-41){$D_2$}
\]

Note that in the illustration of $D_1$ and $D_2$, we use flat crossings for $p_a$ and $p_b$ to represent either type of classical crossing or singular crossing. 
Now choose a base-point on $D_2$ and consider the walk starting at the base-point and apply the same procedure as for $D_1$. This produces a (connected) virtual singular knot diagram, $D_3$.

Continuing with this process, we obtain a finite sequence of virtual singular knot diagrams 
\[D_1,\, D_2,\, D_3, \, \cdots, D_l = D ^{\prime},\] where 
\[ c(D_i) + s(D_i) - c(D_{i+1}) - s(D_{i+1}) = 2 \,\, \text{for all} \,\, 1\leq i \leq l-1, \]
and where any two classical and/or singular crossings $p'_a$ and $p'_b$ in $D^{\prime}$ appear in a list produced by a walk along $D^{\prime}$ as follows:
\[ \cdots p'_a, \cdots, p'_b, \cdots, p'_b, \cdots p'_a ,\cdots  \]
Let $m$ be the number of classical and singular crossings in $D^{\prime}$. Resolving all of these $m$ crossings using the oriented resolution produces a purely virtual magnetic graph with $m+1$ components without bivalent vertices.

It follows that $||S_0|| = m+1$ and that
\begin{eqnarray}\label{lemma-proof}
m = c(D^{\prime}) + s(D^{\prime})  \equiv c(D) + s(D)- (k-1) \pmod 2 .  
\end{eqnarray}
Then,
\begin{eqnarray*}
\text{max}_A(\OR(S_0)) 
&=& 2a(S_0) + 2 \alpha(S_0) + 2 ||S_0||\\
&=& -2w(D) + 2 s(D) + 2(m+1)\\
&=&2(-w(D) + s(D) + m+1).
\end{eqnarray*}
But $w(D) \equiv -c(D) \pmod 2$, and thus
\begin{eqnarray*}
-w(D) + s(D) + m+1
&\equiv& c(D) + s(D) + m +1 \pmod 2 \\
&\stackrel{\eqref{lemma-proof}}{\equiv}& m + k-1 + m +1 \pmod 2 \\
&\equiv& k \pmod 2.
\end{eqnarray*}
Therefore, $\text{max}_A(\,\OR(S_0)) \equiv 2k \pmod 4$. Since all powers of $A$ in $\OR(S_0)$ are congruent to $\text{max}_A(\,\OR(S_0)) $ modulo 4, the desired statement follows.
\end{proof}

\begin{lemma}\label{proof thm-lemma2}
Let $D$ be a virtual singular link diagram and $S$ an enhanced Kauffman-Jones state of $D$. Let $S'$ be an enhanced Kauffman-Jones state of $D$ obtained from $S$ by changing the resolution at one of the classical or singular crossings of $D$.\\
(1) If $||S'|| \neq ||S||$, then $i(S') = i(S)$ and $\text{max}_A(\OR(S')) \equiv \text{max}_A(\OR(S)) \pmod 4$. \\
(2) If $||S'|| = ||S||$, then $i(S')=- i(S)$ and $\text{max}_A(\OR(S')) \equiv \text{max}_A(\OR(S)) + 2 \pmod 4$.
\end{lemma}

\begin{proof}Let $p$ be a (classical or singular) crossing of $D$ where, say, a disoriented resolution occurred to form the state $S$. Then an oriented resolution was applied at $p$ to obtain the state $S'$, and all of the other crossings of $D$ received the same type of resolutions to arrive at $S$ and $S'$, respectively.

(1) Suppose that $||S'|| \neq ||S||$. 
Since $S$ and $S'$ differ only in the neighborhood of $p$, $S'$ either has one less or one  more component than $S$, as demonstrated below (we use dashed lines to represent the regions where $S$ and $S'$ coincide).
\[ 
\raisebox{-.8cm}{\includegraphics[height=.95 in]{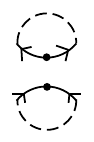}} \put(-25,-35){$S$} 
\put(-34, 25){\fontsize{8}{10}$1$} \put(-21, 25){\fontsize{8}{10}$-1$}
\put(-34, 5){\fontsize{8}{10}$1$} \put(-18, 5){\fontsize{8}{10}$-1$}
\hspace{0.2cm} \raisebox{.2cm}{$\longrightarrow$} \hspace{0.2cm} 
 \raisebox{-.8cm}{\includegraphics[height=.95 in]{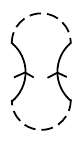}} \put(-22,-35){$S'$} 
\put(-37, 14){\fontsize{8}{10}$1$} \put(-10, 14){\fontsize{8}{10}$-1$}\
\hspace{0.7in}
\raisebox{-.6cm}{\includegraphics[height=.58 in]{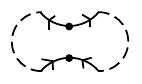}} \put(-43,-35){$S$} 
\put(-47, 20){\fontsize{8}{10}$1$} \put(-33, 20){\fontsize{8}{10}$-1$}
\put(-47, -2){\fontsize{8}{10}$1$} \put(-33, -2){\fontsize{8}{10}$-1$}
\longrightarrow  \raisebox{-.6cm}{\includegraphics[height=.6 in]{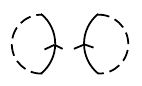}} \put(-40,-35){$S'$}
\put(-54, 5){\fontsize{8}{10}$1$} \put(-26, 5){\fontsize{8}{10}$-1$}
\]

Consider first the case where $||S'|| = ||S|| - 1$. Assign weight maps to $S$ and $S'$ such that the edges shown have weights as depicted in the diagrams above.
Then the weight maps assigned to $S$ and $S'$ coincide on the common regions, and any virtual crossing will have the same parity with respect to $S$ and $S'$. It follows that $i(S)=i(S')$. 

Suppose that $p$ is a singular crossing (the proof is similar for the case when $p$ is a classical  crossing, and thus it is omitted to avoid repetition). We then have the following equalities,
\[ a(S)=a(S'), b(S)=b(S'), ||S||=||S'||+1, \alpha(S)=\alpha(S')-1, \beta(S)=\beta(S')+1,\]
which imply that
\begin{eqnarray*}
\text{max}_A(\OR(S) ) &=& 2a(S) + 4b(S) + 2(\alpha(S) + ||S||) + 4 \beta(S)  \\
&=& 2a(S') + 4b(S') + 2(\alpha(S')-1 + ||S'||+1) + 4 (\beta(S')+1) \\
&=& 2a(S') + 4b(S') + 2(\alpha(S') + ||S'||) + 4 \beta(S') +4 \\
&=& \text{max}_A(\OR(S') ) +4.
\end{eqnarray*}
Thus, $\text{max}_A(\OR(S')) \equiv \text{max}_A(\OR(S)) \pmod 4$. 

 Now consider the case where $||S'|| = ||S|| + 1$, and assign weight maps to $S$ and $S'$ so that the edges shown in the diagrams receive the same weights as in the previous case. Then a given virtual crossing will have the same parity with respect to $S$ and $S'$, respectively, and hence $i(S)=i(S')$. Moreover, suppose this time that $p$ is a positive classical crossing. Then
\[ a(S)=a(S')+1, b(S)=b(S')-1, ||S||=||S'||-1, \alpha(S)=\alpha(S'), \beta(S)=\beta(S'),\]
and therefore,
\begin{eqnarray*}
\text{max}_A(\OR(S) ) &=& 2a(S) + 4b(S) + 2(\alpha(S) + ||S||) + 4 \beta(S)  \\
&=& 2(a(S')+1) + 4(b(S')-1) + 2(\alpha(S')+ ||S'||-1)+ 4 \beta(S') \\
&=& 2a(S') + 4b(S') + 2(\alpha(S') + ||S'||) + 4 \beta(S') -4 \\
&=& \text{max}_A(\OR(S') ) -4.
\end{eqnarray*}
Thus, $\text{max}_A(\OR(S')) \equiv \text{max}_A(\OR(S)) \pmod 4$. If $p$ is a negative classical crossing or a singular crossing, the proof follows similarly. 

(2) Suppose now that $||S'|| = ||S||$, and assign weight maps to $S$ and $S'$ as shown below.
\[ 
\raisebox{-.8cm}{\includegraphics[height=.8 in]{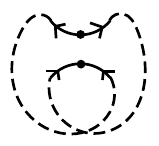}} \put(-33,-35){$S$}
\put(-45, 29){\fontsize{8}{10}$-1$} \put(-23, 29){\fontsize{8}{10}$1$} 
\put(-72, 0){$C_1$} 
\put(0, 0){$C_2$} 
\put(-40, -3){\fontsize{8}{10}$1$} \put(-30, -3){\fontsize{8}{10}$-1$} 
\hspace{0.5cm}
\longrightarrow  
\hspace{0.5cm} \raisebox{-.8cm}{\includegraphics[height=.8 in]{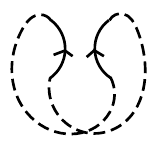}} \put(-35,-35){$S'$} 
\put(-72, 0){$C_1$} 
\put(0, 0){$C_2$} 
\put(-47, 15){\fontsize{8}{10}$1$} \put(-18, 15){\fontsize{8}{10}$1$}
\]

Note that the arcs $C_1$ and $C_2$ (as depicted in the above diagrams) are identical in both of the states $S$ and $S'$, and that these arcs must intersect in an odd number of virtual crossings. In addition, the weight maps for $S$ and $S'$ as defined above exist, since both arcs $C_1$ and $C_2$ must contain an even number of bivalent vertices.

Then, an edge belonging to the arc $C_1$ will have opposite weights with respect to $S$ and $S'$, respectively. On the other hand, an edge belonging to the arc $C_2$ will have the same weights with respect to $S$ and $S'$. It follows that if $v$ is a virtual crossing where arcs $C_1$ and $C_2$ intersect, then the parity of $v$ with respect to $S$ is opposite from the parity of $v$ with respect to $S'$. Since the arcs $C_1$ and $C_2$ intersect in an odd number of virtual crossings, and since the parity of a virtual crossing where $C_1$ (or $C_2$) intersects itself is the same with respect to both states $S$ and $S'$, it follows that $i(S) = - i(S')$.

To show the second part of the statement, suppose that $p$ is a negative classical crossing. Then 
\[ a(S)=a(S')-1, b(S)=b(S')+1, ||S||=||S'||, \alpha(S)=\alpha(S'), \beta(S)=\beta(S').\]
Therefore, 
\begin{eqnarray*}
\text{max}_A(\OR(S) ) &=& 2a(S) + 4b(S) + 2(\alpha(S) + ||S||) + 4 \beta(S)  \\
&=& 2(a(S')-1) + 4(b(S')+1) + 2(\alpha(S')+ ||S'||)+ 4 \beta(S') \\
&=& 2a(S') + 4b(S') + 2(\alpha(S') + ||S'||) + 4 \beta(S') +2 \\
&=& \text{max}_A(\OR(S') ) +2.
\end{eqnarray*}
Thus, $\text{max}_A(\OR(S')) \equiv \text{max}_A(\OR(S))+2 \pmod 4$.
\end{proof}

\begin{lemma}\label{proof thm-lemma3}
Let $D = D_1 \cup D_2$ be the disjoint union of virtual singular link diagrams $D_1$ and $D_2$, and let
$R(D_i)=\phi(D_i)h + \psi(D_i)$ for $i = 1, 2$ as given in Remark~\ref{rem:splitting-def}. Then $R(D) = \phi(D)h + \psi(D)$, where
\[\phi (D)  = \phi(D_1) \psi(D_2) + \psi(D_1)\phi(D_2)  \,\, \text{and} \,\, \psi(D) = \phi(D_1) \phi(D_2) + \psi(D_1) \psi(D_2).   \]
\end{lemma}

\begin{proof} Every enhanced Kauffman-Jones state $S$ of $D$ is of the form $S_1 \cup S_2$, where $S_1$ and $S_2$ are enhanced Kauffman-Jones states of $D_1$ and $D_2$, respectively.
Moreover, since $S_1 \cap S_2 = \emptyset$, the following hold:
\[||S|| = ||S_1|| + ||S_2|| \]
\[ a(S) = a(S_1) + a(S_2), \,\, b(S) = b(S_1) + b(S_2) \]
\[ \alpha(S) = \alpha(S_1) + \alpha(S_2), \,\, \beta(S) = \beta(S_1) + \beta(S_2) \]
\[  i(S) = 1 \Leftrightarrow i(S_1) = i(S_2).\]
We note also that $c(D) = c(D_1) + c(D_2)$. Then the statement follows easily from the defining state-sum formula for the polynomial $R(D)$.
\end{proof}


\textit{Proof of Theorem~\ref{thm:splitting}.} Suppose that $L$ is a virtual singular link with $k$ components and let $D = D_1 \cup D_2 \cup \cdots D_n$ be a diagram of $L$ where $D_i$, for $ 1\leq i \leq n$, is a connected diagram with $k_i$ components, and where $k_1 + \cdots + k_n = k$. We provide a proof by induction on $n$.

If $n=1$, then $D$ is connected and we can use Lemma~\ref{proof thm-lemma1} applied to the oriented enhanced Kauffman-Jones state $S_0$ corresponding to $D$. Thus $\OR(S_0) \in \mathbb{Z}[A^4, A^{-4}] \cdot A^{2k}$. If $S$ is any other enhanced Kauffman-Jones state of $D$, then  $||S|| \neq ||S_0||$ or $||S|| = ||S_0||$. By Lemma ~\ref{proof thm-lemma2}, in the first case $\OR(S) \in \mathbb{Z}[A^4, A^{-4}] \cdot A^{2k}$. In the second case, $\OR(S) \in \mathbb{Z}[A^4, A^{-4}] \cdot A^{2k}\cdot A^2h$, or equivalently, $\OR(S) \in \mathbb{Z}[A^4, A^{-4}] \cdot A^{2(k-1)}h$ (since $2(k+1) \equiv 2(k-1) \pmod 4$). 
Since
\[ R(D) = (-1)^{c(D)} \sum_S \OR(S) 
 = (-1)^{c(D)} \left (\sum_{S, ||S|| = ||S_0||} \OR(S) + \sum_{S, ||S|| \neq ||S_0||} \OR(S)  \right),
\]
it follows that $R(D) = \phi(D) h + \psi(D)$ where $\phi(D) \in \mathbb{Z}[A^4,A^{-4}] \cdot A^{2(k-1)}$
and $\psi(D) \in \mathbb{Z}[A^4,A^{-4}] \cdot A^{2k} $. 

Suppose that $n > 1$ and that the statement is true for any disjoint union of $l$ connected diagrams, where $1\leq l \leq n-1$. Let $D' : = D_1 \cup D_2 \cup \cdots \cup D_{n-1}$ and $D'': = D_n$, and note that $D'$ has $k_1 + \cdots + k_{n-1} = k - k_n$ components, while $D''$ has $k_n$ components. Then $D = D' \cup D''$ and $R(D') = \phi(D') h + \psi(D'), \, R(D'') = \phi(D'') h + \psi(D'')$, where
\[  \phi(D') \in \mathbb{Z}[A^4,A^{-4}] \cdot A^{2(k-k_n-1)}, \,\,  \,\,  \psi(D') \in \mathbb{Z}[A^4,A^{-4}] \cdot A^{2(k-k_n)}, \]
\[  \phi(D'') \in \mathbb{Z}[A^4,A^{-4}] \cdot A^{2(k_n-1)}, \,\,  \,\,  \psi(D'') \in \mathbb{Z}[A^4,A^{-4}] \cdot A^{2k_n}. \]
Then $R(D) = \phi(D)h + \psi(D)$, where by Lemma~\ref{proof thm-lemma3}, 
\begin{eqnarray*}
\phi (D) &=& \phi(D') \psi(D'') + \psi(D')\phi(D'') \in \mathbb{Z}[A^4,A^{-4}] \cdot A^{2(k-1)} \\
\psi(D) &=& \phi(D') \phi(D'') + \psi(D') \psi(D'') \in \mathbb{Z}[A^4,A^{-4}] \cdot A^{2k}.  
\end{eqnarray*}
Therefore, the statement holds for any $n\geq 1$. \hfill $\square$

\begin{corollary}
The extended Kauffman-Jones polynomial of a virtual singular link $L$ with $k$ components decomposes non-trivially into two parts:
\[ \left < \,L\, \right > = \phi(L) + \psi(L),\]
where  $\phi(L) \in \mathbb{Z}[A^4,A^{-4}] \cdot A^{2(k-1)}$ and $\psi(L) \in \mathbb{Z}[A^4,A^{-4}] \cdot A^{2k} $. 
\end{corollary}
\begin{proof}
The statement follows from Theorem ~\ref{thm:splitting} and the fact that $R(L)_{|h = 1} = \left < \,L\, \right >$.
\end{proof}


\begin{thebibliography}{999} 

\bibitem{Jo} V. F. R. Jones, A polynomial invariant for knots via Von Neumann algebras. \textit{Bull. Amer. Math. Soc.}  \textbf{12} (1985), 103-111. 

\bibitem{KM} N. Kamada and Y. Miyazawa, A 2-variable invariant for a virtual link derived from magnetic graphs, {\em Hiroshima Math. J.} \textbf{35}, No. 2 (2005), 309-326. 

\bibitem {Ka1} L. H. Kauffman, State models and the Jones polynomial, {\em Topology} \textbf{26} (1987), 395-407. 

\bibitem {Ka2} L. H. Kauffman, Virtual knot theory, {\em Europ. J. Combinatorics} \textbf{20} (1999), 663-691. 


\end{thebibliography}
\end{document}